\documentclass[twoside]{article}
\usepackage[cp1251]{inputenc}
\usepackage[T2A]{fontenc}
\usepackage{array}
\usepackage{amssymb}
\usepackage{amsmath}
\usepackage{amsthm}
\usepackage{latexsym}
\usepackage{indentfirst}
\usepackage{bm}
\usepackage{enumerate}
\usepackage[dvips]{graphicx}
\usepackage{epsf}
\usepackage{wrapfig}
\usepackage{euscript}
\usepackage{indentfirst}
\usepackage[english,russian]{babel}
\usepackage{textcomp}
\usepackage{enumitem}
\newtheorem{theorem}{Theorem}
\newtheorem{lemma}{Lemma}
\newtheorem{definition}{Definition}

\newtheorem{corollary}{Corollary}

\begin{document}
{\selectlanguage{english}
\binoppenalty = 10000 %
\relpenalty   = 10000 %

\pagestyle{headings} \makeatletter
\renewcommand{\@evenhead}{\raisebox{0pt}[\headheight][0pt]{\vbox{\hbox to\textwidth{\thepage\hfill \strut {\small Grigory K. Olkhovikov}}\hrule}}}
\renewcommand{\@oddhead}{\raisebox{0pt}[\headheight][0pt]{\vbox{\hbox to\textwidth{{Restricted interpolation in stit logic}\hfill \strut\thepage}\hrule}}}
\makeatother

\title{Restricted Interpolation and Lack Thereof in Stit Logic}
\author{Grigory K. Olkhovikov\\Ruhr University Bochum\\
Department of Philosophy II; NAFO 02/299\\
Universit\"{a}tstr. 150, D-44780, Bochum, Germany\\
email: grigory.olkhovikov@rub.de, grigory.olkhovikov@gmail.com}
\date{}
\maketitle
\begin{quote}
\textbf{Abstract}. We consider the propositional logic equipped
with \emph{Chellas stit} operators for a finite set of individual
agents plus the historical necessity modality. We settle the
question of whether such a logic enjoys restricted interpolation
property, which requires the existence of an interpolant only in
cases where the consequence contains no Chellas stit operators
occurring in the premise. We show that if action operators count
as logical symbols, then such a logic has restricted interpolation
property iff the number of agents does not exceed three. On the
other hand, if action operators are considered to be non-logical
symbols, the restricted interpolation fails for any number of
agents exceeding one. It follows that unrestricted Craig
interpolation also fails for almost all versions of stit logic.
\end{quote}

\begin{quote}
\end{quote}

\section{Introduction}

The so-called stit logic is the modal logic of actions that uses
the locution `$j$ sees to it that $A$' (where $j$ is an agent name
and $A$ a sentence) as its paradigm of action modality. The very
name `stit' derives from the acronym of this paradigm locution.
This logic has been present and explored in the literature on
philosophical logic at least since the 1980s. Many of the early
defining texts in the stit tradition were authored and coauthored
by N. Belnap, and the book \cite{belnap2001facing} is a useful
guide to the early steps of this type of research and its
attending controversies. However, in \cite{belnap2001facing} N.
Belnap comes forward as a proponent of the so-called
\emph{achievement} stit operator, whereas the later work in stit
logic mainly concentrated around the \emph{Chellas} stit and
\emph{deliberative} stit operators.\footnote{Chellas stit is named
after B. Chellas, who introduced a similar operator in
\cite{chellas}. These two stit operators are interdefinable in the
presence of historical necessity modality; therefore, one is
inclined to say that they share the same logic. Chellas stit
operator is somewhat simpler and often used as the basic one,
whereas the deliberative stit is often defined in terms of Chellas
stit.} Deliberative stit operator was independently proposed by F.
von Kutschera (see, e.g. \cite{vK1986}) and J. Horty (see, e.g.,
\cite{horty1995}). The present paper follows this line so that the
name of stit logic gets applied to the logic of Chellas
stit/deliberative stit operator with Chellas stit taken as the
basic stit operator, and deliberative stit as the defined one.

Most of the work on stit logic since these early days had a
conceptual focus, applying stit semantics to modelling
philosophical questions and exploring alternative stit operators
which were proposed as improved versions of achievement and
deliberative stit in some respect (see, e.g., \cite{broersen}).
More recently emerged the attempts to enrich stit logic with other
types of operators, e.g. the ones borrowed from temporal logic
(see, e.g., \cite{lorini}) or justification logic (see, e.g.,
\cite{OLWA} and \cite{OLWA2}). Sometimes these attempts were
intertwined with attempts to recast the stit semantics itself so
as to make it more suitable for the enrichment in question.

As for the more technical work on stit logic, it mostly
concentrated on forging axiomatizations and, to some extent,
solving the computational complexity questions. Some of the
relatively recent important contributions to this research are
e.g. \cite{HerzigSchwarzent} and \cite{balbiani}.

One of the standard refinements of completeness results is the
Craig Interpolation Property. However, to the best of our
knowledge, this direction of research in stit logic has yet to see
its first contributions. We hope that our paper will be able to
cover this gap at least to some extent. The paper mainly focuses
on a restriction of the Craig Interpolation Property which only
requires existence of an interpolant if the antecedent shares no
agent names with the consequent. However, we show that even this
weakened version of interpolation property fails for stit logic if
the logic deals with more than three different agents. Of course,
the failure of restricted Craig interpolation entails also the
failure of the unrestricted interpolation property. Therefore, an
easy corollary to the main result of this paper is the failure of
unrestricted Craig interpolation in stit logic for any number of
agents exceeding three, which yields the negative solution to the
problem of Craig interpolation for the vast majority of variants
of the basic stit logic.

We now briefly touch upon the structure of the text below. Section
\ref{preliminaries} defines the version of stit logic at hand in
terms of language, semantics, and a strongly complete
axiomatization. We also introduce the main notations to be used in
the paper and give the precise definition of the Restricted Craig
Interpolation Property for stit logic of $n$ agents. The latter
property will be the main subject of the two following sections.
We are going to show, first, that whenever our version of stit
logic has no more than three different agents, it enjoys this
property. The proof of this positive part of our main result is
given in Section \ref{S:positive}. The corresponding negative
part, saying that the Restricted Craig Interpolation Property
fails for stit logic with more than three agents, is then
formulated and proven in Section \ref{S:negative}. After that,
Section \ref{S:further} explores the various corollaries of the
main result in relation to the following topics: (a) unrestricted
Craig interpolation, (b) the Restricted Robinson Consistency
Property, and (c) the stronger versions of both unrestricted and
restricted interpolation property which treat stit operators as
non-logical symbols.

Section \ref{conclusion} sums up the preceding sections and charts
some natural continuations for the line of research presented in
the paper.

\section{Preliminaries}\label{preliminaries}

On the basis of a given a finite agent community $Ag$ and a set of
propositional variables $V$, we define the set
$\mathcal{L}^{Ag}_{V}$ of $(Ag,V)$-stit formulas as follows:
$$
A := p \mid A \to A \mid \bot \mid \Box A \mid [j]A,
$$
where $p \in V$ and $j \in Ag$. Stit formulas will be denoted by
letters $A$, $B$, $C$, $D$, decorated with sub- and superscripts
whenever needed. Formulas of the type $\Box A$ and $[j]A$ are
informally read as `$A$ is (historically) necessary' and `the
agent $j$ sees to it that $A$', respectively. We reserve $\Diamond
A$ and $\langle j\rangle A$ as the notations for the duals of
these modalities.

Modalities of the form $[j]$ for $j \in Ag$ are called
\emph{action modalities} and will be interpreted as Chellas stit
operators for the respective agent $j$. We will not use
deliberative stit operator $[d:j]$ in this paper, but it can be
defined on the basis of Chellas stit and historical necessity:
$[d:j]A := [j]A \wedge \neg\Box A$. Although $Ag$ is normally
assumed to be non-empty, in this paper we will allow for $Ag =
\emptyset$ as a border case for the sake of notational
convenience. The set $\mathcal{L}^{\emptyset}_{V}$ is then
basically a variant of the language of the logic of historical
necessity. This logic is known to coincide with propositional S5
and hence has Craig Interpolation Property.\footnote{In fact,
propositional S5 even has the stronger Lyndon interpolation
property, see e.g \cite[Theorem 5.14, p. 140]{gabbay}.} Therefore,
even though empty agent communities are allowed by our notation,
we will not consider interpolation properties of the languages
devoid of action modalities in this paper.

Stit formulas are interpreted over the respective classes of stit
models. An $(Ag,V)$-\emph{stit model} is a structure of the form
$\mathfrak{S} = \langle Tree, \leq, Choice, V\rangle$, such that:
\begin{itemize}
\item $Tree$ is a non-empty set. Elements of $Tree$ are called
\emph{moments}.

\item $\leq$ is a partial order on $Tree$ for which a temporal
interpretation is assumed.

\item $Hist(Tree, \leq)$ is the set of maximal chains in $Tree$
w.r.t. $\leq$. Since $Hist(Tree, \leq)$ is completely determined
by $Tree$ and $\leq$, it is not included into the structure of a
model as a separate component. Elements of $Hist(Tree, \leq)$ are
called \emph{histories}. The set of histories containing a given
moment $m$ will be denoted $H^\mathfrak{S}_m$. The following set
$$
MH(Tree, \leq) = \{ (m,h)\mid m \in Tree,\, h \in H^\mathfrak{S}_m
\},
$$
called the set of \emph{moment-history pairs}, will be used to
evaluate formulas in $\mathcal{L}^{Ag}_V$.

Two histories, $h, g \in H^\mathfrak{S}_m$ we call
\emph{undivided} at $m \in Tree$ and write $h \approx_m g$ iff $h$
and $g$ share some later moment $m'$. In other words, we stipulate
that:

$$
h \approx_m g \Leftrightarrow (h, g \in H^\mathfrak{S}_m)
\&(\exists m' > m)(h, g \in H^\mathfrak{S}_{m'}).
$$

\item $Choice$ is a function mapping $Tree \times Ag$ into
$2^{2^{Hist(Tree, \leq)}}$ in such a way that for any given $j \in
Ag$ and $m \in Tree$ we have as $Choice(m,j)$ (to be denoted as
$Choice^m_j$ below) a partition of $H^\mathfrak{S}_m$. For a given
$h \in H^\mathfrak{S}_m$ we will denote by $Choice^m_j(h)$ the
element of the partition $Choice^m_j$ (otherwise called a
\emph{choice cell}) containing $h$. Intuitively, the idea is that
$j$ cannot distinguish by her activity at $m$ between histories
that belong to one and the same choice cell.

\item $V$ is an evaluation function, mapping the set $V$ into
$2^{MH(Tree, \leq)}$
\end{itemize}
In what follows, for a given $(Ag,V)$-stit model  $\mathfrak{S} =
\langle Tree, \leq, Choice, V\rangle$, we will sometimes use
$Hist(\mathfrak{S})$ and $MH(\mathfrak{S})$ to denote $Hist(Tree,
\leq)$ and $MH(Tree, \leq)$, respectively.

Additionally, every stit model $\mathfrak{S}$ is required to
satisfy the following constraints:
\begin{enumerate}
\item \textbf{Historical connection}:
\begin{equation}\label{HC}\tag{\text{HC}}
(\forall m,m_1 \in Tree)(\exists m_2 \in Tree)(m_2 \leq m\, \&\,
m_2 \leq  m_1)
\end{equation}
\item \textbf{No backward branching}:
\begin{equation}\label{NBB}\tag{\text{NBB}}
(\forall m,m_1,m_2 \in Tree)((m_1 \leq m\, \&\, m_2 \leq m)
\Rightarrow (m_1 \leq m_2 \vee m_2 \leq m_1))
\end{equation}
\item \textbf{No choice between undivided histories}:
\begin{equation}\label{NCUH}\tag{\text{NCUH}}
(\forall m \in Tree)(\forall h,h' \in H^\mathfrak{S}_m)(h
\approx_m h' \Rightarrow Choice^m_j(h) = Choice^m_j(h'))
\end{equation}
for every $j \in Ag$.

\item \textbf{Independence of agents}:
\begin{equation}\label{IA}\tag{\text{IA}}
(\forall f:Ag \to 2^{H^\mathfrak{S}_m})((\forall j \in Ag)(f(j)
\in Choice^m_j) \Rightarrow \bigcap_{j \in Ag}f(j) \neq \emptyset)
\end{equation}
for every $m\in Tree$.
\end{enumerate}
We omit the motivation for these constraints, referring the reader
to the existing literature on stit logic, e.g.
\cite{belnap2001facing} and \cite{horty2001agency}. The inductive
definition of the satisfaction relation for the members of
$\mathcal{L}^{Ag}_V$ is then as follows:
\begin{align*}
&\mathfrak{S}, m, h \models p \Leftrightarrow (m,h) \in
V(p);\\
&\mathfrak{S}, m, h \models [j]A \Leftrightarrow (\forall h'
\in Choice^m_j(h))(\mathfrak{S}, m, h' \models A);\\
&\mathfrak{S}, m, h \models \Box A \Leftrightarrow (\forall h' \in
H^\mathfrak{S}_m)(\mathfrak{S}, m, h' \models A),
\end{align*}
with the usual clauses for the Boolean connectives. The notions of
satisfaction and validity are also defined in a standard way.

Stit logic, as given above, admits of the following strongly
complete axiomatization $\mathbb{S}$ which we borrow from
\cite{balbiani}.\footnote{The original proof, due to Ming Xu, used
a somewhat more expressive language allowing also to describe
equality/inequality relations between agents, see e.g. \cite[Ch.
17]{belnap2001facing}.} The axiom schemes of $\mathbb{S}$ are as
follows:
\begin{align}
&\textup{A full set of axioms for classical propositional
logic}\label{A0}\tag{\text{A0}}\\
&\textup{$S5$ axioms for $\Box$ and $[j]$ for every $j \in
Ag$}\label{A1}\tag{\text{A1}}\\
&\Box A \to [j]A \textup{ for every }j \in Ag\label{A2}\tag{\text{A2}}\\
&(\Diamond[j_1]A_1 \wedge\ldots \wedge \Diamond[j_n]A_n) \to
\Diamond([j_1]A_1 \wedge\ldots
\wedge[j_n]A_n)\label{A3}\tag{\text{A3}}
\end{align}
The assumption is that in \eqref{A3} $j_1,\ldots, j_n$ are
pairwise different.

In addition to the axioms, $\mathbb{S}$ contains two inference
rules:
\begin{align}
&\textup{From }A, A \to B \textup{ infer } B;\label{MP}\tag{\text{MP}}\\
&\textup{From }A\textup{ infer } \Box
A;\label{Nec}\tag{\text{Nec}}
\end{align}
Provability of $A$ in $\mathbb{S}$ we will denote by $\vdash A$.
It is clear that the strong completeness of $\mathbb{S}$ also
implies compactness of stit logic for any given finite community
$Ag$ of agents and any given set $V$ of propositional variables.

We introduce some further useful notations related to sets of stit
formulas. If $\Gamma \subseteq \mathcal{L}^{Ag}_V$, then we let
$\Gamma^\Box$ denote the set of all boxed formulas from $\Gamma$.
Similarly, whenever $j \in Ag$, we use $\Gamma^{[j]}$ to denote
the set $\{ [j]A \in \mathcal{L}^{Ag}_V\mid [j]A \in \Gamma\}$.

For arbitrary $Ag$, $V$, and a set $\Gamma \cup \{ A\} \subseteq
\mathcal{L}^{Ag}_V$, we extend the notation $\vdash$ to contexts
like $\Gamma \vdash A$ to mean that $\vdash (A_1 \wedge\ldots
\wedge A_r) \to A$ for some $A_1,\ldots, A_r \in \Gamma$. Then
$\Gamma \subseteq \mathcal{L}^{Ag}_V$ is called
\emph{inconsistent} iff $\Gamma \vdash \bot$, and
\emph{consistent} otherwise. Moreover, $\Gamma \subseteq
\mathcal{L}^{Ag}_V$ is $(Ag,V)$-\emph{maxiconsistent} iff it is
consistent and no consistent subset of $\mathcal{L}^{Ag}_V$
properly extends $\Gamma$. It can be shown, in the usual way, that
an arbitrary $\Gamma \subseteq \mathcal{L}^{Ag}_V$ is
$(Ag,V)$-maxiconsistent iff for every $A \in \mathcal{L}^{Ag}_V$
the set $\Gamma \cap \{ A, \neg A \}$ is a singleton. In what
follows we will need the following classical lemma about
maxiconsistent sets:
\begin{lemma}\label{maxiconsistent}
For any finite $Ag$ and any set of propositional variables $V$, if
$\Gamma \subseteq \mathcal{L}^{Ag}_V$ is consistent but not
maxiconsistent, then there is an $A \in \mathcal{L}^{Ag}_V$ such
that $\{ A, \neg A\}\cap \Gamma = \emptyset$.
\end{lemma}
\begin{proof} If $\Gamma \subseteq \mathcal{L}^{Ag}_V$ is consistent but not
maxiconsistent, then choose a consistent $\Xi$ such that $\Gamma
\subset \Xi \subseteq \mathcal{L}^{Ag}_V$ and choose any $A \in
\Xi \setminus \Gamma$. Then $A \notin \Gamma$ by choice of $A$,
and if $\neg A \in \Gamma$, then $\{ A, \neg A\} \subseteq \Gamma
\cup \{ A \} \subseteq \Xi$, which contradicts the consistency of
$\Xi$ since, of course, $\vdash (A \wedge \neg A) \to \bot$.
Therefore, we must also have $\neg A \notin \Gamma$ so that $\{ A,
\neg A\}\cap \Gamma = \emptyset$.
\end{proof}

For a $\Gamma \subseteq \mathcal{L}^{Ag}_V$ we define that:
$$
|\Gamma| := \{ p \in V \mid p\textup{ occurs in }\Gamma \},
$$
and:
$$
Ag(\Gamma) := \{ j \in Ag \mid j\textup{ occurs in }\Gamma \},
$$
If $\Gamma$ is a singleton $\{A \}$, then we use the notations
$|A|$ and $Ag(A)$ instead of $|\{A\}|$ and $Ag(\{A\})$.

In this paper we will be mainly testing the applicability to stit
logic of the following property:
\begin{definition}\label{D:rcip}
For a positive integer $n$, stit logic has the \emph{Restricted}
$n$-\emph{Craig Interpolation Property} (abbreviated by
$(RCIP)_n$) iff for any set of propositional variables $V$, and
all $A, B \in \mathcal{L}^{\{1,\ldots, n\}}_V$, whenever $\vdash A
\to B$ and $Ag(A) \cap Ag(B) = \emptyset$, then there exists a $C
\in \mathcal{L}^{Ag(A) \cup Ag(B)}_{|A| \cap |B|}$ such that both
$\vdash A \to C$ and $\vdash C \to B$.
\end{definition}

\section{The case $n \leq 3$}\label{S:positive}
The main result of this section looks as follows:
\begin{theorem}\label{positive}
For every $n \leq 3$, stit logic has $(RCIP)_n$.
\end{theorem}

We prepare the result by proving several technical lemmas first.
\begin{lemma}\label{technical}
The following statements are true:

1. For every agent index $j$, $[j]$ is an S5-modality.

2. Let $A, B_1,\ldots B_n, C \in \mathcal{L}^{Ag}_V$, let $i_1,
\ldots, i_n, j \in Ag$ be pairwise different, and let $\vdash
(\Box A \wedge [i_1]B_1 \wedge\ldots\wedge [i_n]B_n) \to \neg C$.
Then also $\vdash (\Box A \wedge \Diamond[i_1]B_1
\wedge\ldots\wedge \Diamond[i_n]B_n) \to \neg\Diamond[j]C$.

3. Let $A, B, C \in \mathcal{L}^{Ag}_V$, let $j \in Ag$, and let
$\vdash (\Box A \wedge [j]B) \to C$. Then also $\vdash (\Box A
\wedge \Diamond[j]B) \to \Diamond[j]C$.
\end{lemma}
\begin{proof}
(Part 1). Immediately by \eqref{A1}, \eqref{Nec}, and \eqref{A2}.

(Part 2). Assume the hypothesis of Part 2 and assume that we have:
\begin{equation}\label{E:e1}
    \vdash (\Box A \wedge [i_1]B_1 \wedge\ldots\wedge [i_n]B_n) \to \neg
C
\end{equation}
Then we reason as follows:
\begin{align}
&\vdash (\Box A \wedge (\Diamond[i_1]B_1 \wedge\ldots\wedge
\Diamond[i_n]B_n \wedge \Diamond[j]C)) \to\notag\\
&\qquad\qquad\to(\Box A \wedge \Diamond([i_1]B_1
\wedge\ldots\wedge[i_n]B_n \wedge[j]C))\label{E:e2} &&\text{(by
\eqref{A3})}\\
&\vdash (\Box A \wedge \Diamond([i_1]B_1 \wedge\ldots\wedge
[i_n]B_n \wedge [j]C)) \to\notag\\
&\qquad\qquad\to(\Box A \wedge \Diamond([i_1]B_1
\wedge\ldots\wedge[i_n]B_n \wedge C))\label{E:e3} &&\text{(by
\eqref{A1})}\\
&\vdash (\Box A \wedge \Diamond([i_1]B_1 \wedge\ldots\wedge
[i_n]B_n \wedge C)) \to\notag\\
&\qquad\qquad\to\Diamond(\Box A \wedge [i_1]B_1
\wedge\ldots\wedge[i_n]B_n \wedge C)\label{E:e4}
&&\text{($\Box$ is S5)}\\
&\vdash (\Box A \wedge (\Diamond[i_1]B_1 \wedge\ldots\wedge
\Diamond[i_n]B_n \wedge \Diamond[j]C)) \to\notag\\
&\qquad\qquad\to\Diamond(\Box A \wedge [i_1]B_1
\wedge\ldots\wedge[i_n]B_n \wedge C)\label{E:e5} &&\text{(by
\eqref{E:e2}-\eqref{E:e4})}\\
&\vdash\Box((\Box A \wedge [i_1]B_1 \wedge\ldots\wedge [i_n]B_n)
\to \neg C)\label{E:e6} &&\text{(by \eqref{E:e1} and \eqref{Nec})}\\
&\vdash\neg\Diamond(\Box A \wedge [i_1]B_1
\wedge\ldots\wedge[i_n]B_n \wedge C)\label{E:e7} &&\text{(by \eqref{E:e6} and prop. logic)}\\
&\vdash\neg(\Box A \wedge (\Diamond[i_1]B_1 \wedge\ldots\wedge
\Diamond[i_n]B_n \wedge \Diamond[j]C))\label{E:e8} &&\text{(by
\eqref{E:e5} and \eqref{E:e7})}
\end{align}
From \eqref{E:e8}, it follows by propositional logic that $\vdash
(\Box A \wedge \Diamond[i_1]B_1 \wedge\ldots\wedge
\Diamond[i_n]B_n) \to \neg\Diamond[j]C$.

(Part 3). We reason as follows:
\begin{align}
&\vdash (\Box A \wedge [j]B) \to C\label{E:p3-1} &&\text{(premise)}\\
&\vdash [j]((\Box A \wedge [j]B) \to C)\label{E:p3-2}&&\text{(by \eqref{E:p3-1} and Part 1)}\\
&\vdash ([j]\Box A \wedge [j]B) \to [j]C\label{E:p3-3}&&\text{(by \eqref{E:p3-2} and Part 1)}\\
&\vdash \Box A \to \Box\Box A\label{E:p3-4}&&\text{(by \eqref{A1})}\\
&\vdash \Box\Box A \to [j]\Box A\label{E:p3-5}&&\text{(by \eqref{A2})}\\
&\vdash \Box A \to [j]\Box A\label{E:p3-6}&&\text{(by \eqref{E:p3-4} and \eqref{E:p3-5})}\\
&\vdash (\Box A \wedge [j]B) \to [j]C\label{E:p3-7}&&\text{(by \eqref{E:p3-3} and \eqref{E:p3-6})}\\
&\vdash (\Box A \wedge \Diamond[j]B) \to
\Diamond[j]C\label{E:p3-8}&&\text{(by \eqref{E:p3-7} and S5
properties of $\Box$)}
\end{align}
\end{proof}

Assume that $V$ is a set of propositional variables and $Ag$ a
finite community of agents. A pair $(\Gamma,\Delta)$ of sets of
$(Ag, V)$-stit formulas, is called \emph{inseparable}, iff
$Ag(\Gamma) \cap Ag(\Delta) = \emptyset$, and for no $A \in
\mathcal{L}^{Ag(\Gamma) \cup Ag(\Delta)}_{|\Gamma| \cap |\Delta
|}$ it is true that both $\Gamma \vdash A$ and $\Delta \vdash \neg
A$. Below we basically repeat the classical argument for the proof
of the following standard lemma about inseparability:
\begin{lemma}\label{inseparable1}
Let $(\Gamma,\Delta)$ be an inseparable pair, and assume that both
$|\Gamma|$ and $|\Delta |$ are at most countable.\footnote{This
lemma also holds for uncountable sets of variables but we will not
need this more general version in the present paper.} Then:
\begin{enumerate}
\item There exist $\Gamma'$ and $\Delta'$ such that $\Gamma
\subseteq \Gamma' \subseteq \mathcal{L}^{Ag(\Gamma)}_{|\Delta |}$,
$\Delta \subseteq \Delta' \subseteq
\mathcal{L}^{Ag(\Delta)}_{|\Delta |}$, $(\Gamma',\Delta')$ is
inseparable, $\Gamma'$ is $(Ag(\Gamma), |\Gamma|)$-maxiconsistent,
and $\Delta'$ is $(Ag(\Delta), |\Delta|)$-maxiconsistent.

\item If $\Gamma' \subseteq \Gamma$ and $\Delta' \subseteq
\Delta$, then $(\Gamma',\Delta')$ is inseparable.
\end{enumerate}
\end{lemma}
\begin{proof} (Part 1) We proceed as in the case of classical
logic. We first enumerate the formulas in
$\mathcal{L}^{Ag(\Gamma)}_{|\Gamma |}$ as $A_0,\ldots,A_s,\ldots,$
and the formulas in $\mathcal{L}^{Ag(\Delta)}_{|\Delta |}$ as
$B_0,\ldots,B_s,\ldots,$. We then define two increasing sequences
of sets of formulas:
$$
\Gamma = \Gamma_0 \subseteq \ldots \subseteq \Gamma_s \subseteq
\ldots
$$
and
$$
\Delta = \Delta_0 \subseteq \ldots \subseteq \Delta_s \subseteq
\ldots
$$
in $\mathcal{L}^{Ag(\Gamma)}_{|\Gamma |}$ and
$\mathcal{L}^{Ag(\Delta)}_{|\Delta |}$, respectively. The
definition is as follows. $\Gamma_0$ and $\Delta_0$ are just
$\Gamma$ and $\Delta$, and whenever $\Gamma_r$ and $\Delta_r$ are
defined for an $r \in \omega$, then we set:
$$
\Gamma_{r + 1} = \left\{%
\begin{array}{ll}
    \Gamma_r \cup \{ A_r\}, & \hbox{if $(\Gamma_r \cup \{ A_r\}, \Delta_r)$ is inseparable;} \\
    \Gamma_r, & \hbox{otherwise.} \\
\end{array}%
\right.
$$
and, further:
$$
\Delta_{r + 1} = \left\{%
\begin{array}{ll}
    \Delta_r \cup \{ B_r\}, & \hbox{if $(\Gamma_{r+1} , \Delta_r\cup \{ B_r\})$ is inseparable;} \\
    \Delta_r, & \hbox{otherwise.} \\
\end{array}%
\right.
$$

\emph{Claim 1}. For every $r \in \omega$, the pairs $(\Gamma_r,
\Delta_r)$ and $(\Gamma_{r + 1}, \Delta_r)$ are inseparable.

The Claim is proved by induction on $r$. If $r = 0$ then
$(\Gamma_0, \Delta_0) = (\Gamma, \Delta)$ is inseparable by the
assumption of the lemma, and the inseparability of $(\Gamma_1,
\Delta_0)$ follows by the definition of $\Gamma_1$. If $r = s +
1$, then $(\Gamma_{s+1}, \Delta_s)$ is inseparable by the
induction hypothesis, whence the inseparability of $(\Gamma_{s+1},
\Delta_{s+1})$ follows by the definition of $\Delta_{s+1}$. From
the latter, the inseparability of $(\Gamma_{s+2}, \Delta_{s+1})$
follows by the definition of $\Gamma_{s+2}$. Claim 1 is proved.

We now set:
$$
\Gamma' := \bigcup_{s \in \omega}\Gamma_s;\qquad \Delta' :=
\bigcup_{s \in \omega}\Delta_s.
$$
We clearly have both:
\begin{equation}\label{E:u1}
\Gamma \subseteq \Gamma_1 \subseteq \ldots \subseteq \Gamma_s
\subseteq \ldots\subseteq \Gamma' \subseteq
\mathcal{L}^{Ag(\Gamma)}_{|\Gamma |}
\end{equation}
and:
\begin{equation}\label{E:u2}
\Delta \subseteq \Delta_1 \subseteq \ldots \subseteq \Delta_s
\subseteq \ldots \subseteq \Delta' \subseteq
\mathcal{L}^{Ag(\Delta)}_{|\Delta |}
\end{equation}
We now show a series of further claims:

\emph{Claim 2}. The sets $\Gamma'$, $\Delta'$ are consistent

Indeed, if $\Gamma'$ is inconsistent then $\vdash A_{t_1} \wedge
\ldots \wedge A_{t_r} \to \bot$ for some $A_{t_1},\ldots, A_{t_r}$
in the above enumeration of $\mathcal{L}^{Ag(\Gamma)}_{|\Gamma |}$
such that $A_{t_1},\ldots, A_{t_r} \in \Gamma'$. Then, by
definition of $\Gamma'$, we must also have $A_{t_1},\ldots,
A_{t_r} \in \Gamma_s$, where $s = max(t_1,\ldots, t_r) + 1$ so
that we have $\Gamma_s \vdash \bot$. Of course, we also have
$\Delta_s \vdash \neg\bot$, and since $\bot \in
\mathcal{L}^{Ag(\Gamma) \cup Ag(\Delta)}_{|\Gamma |\cap
|\Delta|}$, it follows that $(\Gamma_s, \Delta_s)$ is separable, a
contradiction to Claim 1. Therefore, $\Gamma'$ is consistent, and
the consistency of $\Delta'$ is established in a similar way.

\emph{Claim 3}. The sets $\Gamma'$, $\Delta'$ are $(Ag(\Gamma),
|\Gamma|)$-maxiconsistent, and $(Ag(\Delta),
|\Delta|)$-maxiconsistent, respectively.

Indeed, if $\Gamma'$ is not  $(Ag(\Gamma),
|\Gamma|)$-maxiconsistent, then it follows from Claim 2 and Lemma
\ref{maxiconsistent}, that there is an $A \in
\mathcal{L}^{Ag(\Gamma)}_{|\Gamma |}$ such that $\{ A, \neg A \}
\cap \Gamma' = \emptyset$. Then we will have $A = A_r$ and $\neg A
= A_{r'}$ for some $r,r' \in \omega$ in terms of our enumeration
of $\mathcal{L}^{Ag(\Gamma)}_{|\Gamma |}$. Since $A_r, A_{r'}
\notin \Gamma'$ we will have, by definition of $\Gamma'$, that
$(\Gamma_r \cup \{ A_r\}, \Delta_r)$ and $(\Gamma_{r'} \cup \{
A_{r'}\}, \Delta_{r'})$ are separable. This means that there exist
some $A^r_1,\ldots, A^r_{t_1} \in \Gamma_r$, $A^{r'}_1,\ldots,
A^{r'}_{t_2} \in \Gamma_{r'}$, $B^r_1,\ldots, B^r_{t_3} \in
\Delta_r$, $B^{r'}_1,\ldots, B^{r'}_{t_4} \in \Delta_{r'}$, and
$C, D \in \mathcal{L}^{Ag(\Gamma) \cup Ag(\Delta)}_{|\Gamma |\cap
|\Delta|}$ such that all of the following holds:
\begin{align}
&\vdash (A^r_1\wedge \ldots\wedge A^r_{t_1} \wedge A) \to C\label{E:f1}\\
&\vdash (B^r_1\wedge \ldots\wedge B^r_{t_3}) \to \neg C\label{E:f2}\\
&\vdash (A^{r'}_1\wedge \ldots\wedge A^{r'}_{t_2} \wedge \neg A) \to D\label{E:f3}\\
&\vdash (B^{r'}_1\wedge \ldots\wedge B^{r'}_{t_4}) \to \neg
D\label{E:f4}
\end{align}
We then infer, by propositional logic, that:
\begin{align}
&\vdash (\bigwedge^{t_1}_{s = 1}A^r_s\wedge \bigwedge^{t_2}_{s = 1}A^{r'}_s) \to (C \vee D)\label{E:f5}\\
&\vdash (\bigwedge^{t_3}_{s = 1}B^r_s\wedge \bigwedge^{t_4}_{s =
1}B^{r'}_s) \to \neg(C \vee D)\label{E:f6}
\end{align}
Now set $r'' := max(r, r')$. By \eqref{E:u1} and \eqref{E:u2} we
know that $\{ A^r_s \mid 1 \leq s \leq t_1 \}\cup \{ A^{r'}_s \mid
1 \leq s \leq t_2 \} \subseteq \Gamma_{r''}$ and that $\{ B^r_s
\mid 1 \leq s \leq t_3 \}\cup \{ B^{r'}_s \mid 1 \leq s \leq t_4
\} \subseteq \Delta_{r''}$. We also clearly have that $C \vee D
\in \mathcal{L}^{Ag(\Gamma) \cup Ag(\Delta)}_{|\Gamma |\cap
|\Delta|}$. Therefore, it follows from \eqref{E:f5} and
\eqref{E:f6} that $(\Gamma_{r''}, \Delta_{r''})$ is separable, in
contradiction to Claim 1. Therefore, $\Gamma'$ must be
$(Ag(\Gamma), |\Gamma|)$-maxiconsistent. Maxiconsistency of
$\Delta'$ is shown in a similar way.

\emph{Claim 4}. $(\Gamma', \Delta')$ is inseparable.

Since $\Gamma'$, $\Delta'$ are maxiconsistent, they are closed for
finite conjunctions. Therefore, we can assume wlog, that there are
$A \in \Gamma'$, $B \in \Delta'$ and $C \in
\mathcal{L}^{Ag(\Gamma) \cup Ag(\Delta)}_{|\Gamma |\cap |\Delta|}$
such that all of the following holds:
\begin{align}
&\vdash A \to C\label{E:f7}\\
&\vdash B \to \neg C\label{E:f8}
\end{align}
Then let $r,s \in \omega$ be such that $A \in \Gamma_r$ and $B \in
\Delta_s$. Setting $t := max(r,s)$, we know that $A \in \Gamma_t$
and $B \in \Delta_t$ whence it follows that $(\Gamma_t, \Delta_t)$
is separable, in contradiction to Claim 1.

Claims 2--4 then imply the first part of the Lemma.

(Part 2). Immediate from the definition of separability.
\end{proof}
\begin{lemma}\label{L:separability}
If $(\Gamma,\Delta)$ is separable then for some finite $\Gamma'
\subseteq \Gamma$ and $\Delta' \subseteq \Delta$ the pair
$(\Gamma',\Delta')$ is also separable.
\end{lemma}
\begin{proof}
If $(\Gamma,\Delta)$ is separable then for some $A \in
\mathcal{L}^{Ag(\Gamma) \cup Ag(\Delta)}_{|\Gamma| \cap |\Delta
|}$ it is true that both $\Gamma \vdash A$ and $\Delta \vdash \neg
A$. By definition, this means that there are $A_1,\ldots, A_r \in
\Gamma$ and $B_1,\ldots, B_s \in \Delta$ such that both $\vdash
(A_1\wedge \ldots\wedge A_r) \to A$ and $\vdash (B_1\wedge
\ldots\wedge B_s) \to\neg A$. Therefore, we can set $\Gamma' := \{
A_1,\ldots, A_r \}$ and $\Delta' := \{ B_1,\ldots, B_s \}$.
\end{proof}
Next we prove two lemmas which sum up some important facts about
inseparability that are peculiar to stit logic:
\begin{lemma}\label{inseparable2}
Let $V$ be a set of propositional variables, let $n \leq 3$, and
let $\Gamma, \Delta \subseteq \mathcal{L}^{\{ 1,\ldots, n \}}_V$
be such that $(\Gamma,\Delta)$ is inseparable. Moreover, assume
that $\Gamma$ is $(Ag(\Gamma), |\Gamma|)$-maxiconsistent and
$\Delta$ is $(Ag(\Delta), |\Delta|)$-maxiconsistent. Finally,
assume that there exist $\Diamond[j_1]A_1,\ldots,\Diamond[j_r]A_r
\in \Gamma$, and $\Diamond[i_1]B_1,\ldots,\Diamond[i_s]B_s \in
\Delta$  such that $j_1,\ldots, j_r \in Ag(\Gamma)$ are pairwise
different and $i_1,\ldots, i_s \in Ag(\Delta)$ are pairwise
different.

Then the pair:
\begin{equation}\label{pair}
(\Gamma^\Box \cup \{ [j_1]A_1,\ldots,[j_r]A_r\}, \Delta^\Box\cup\{
[i_1]B_1,\ldots,[i_s]B_s\})
\end{equation}
is inseparable.
\end{lemma}
\begin{proof} Assume the hypothesis, and assume, for \emph{reductio},
that \eqref{pair} is separable. Then, by compactness of stit logic
and the S5 properties of $\Box$, there must be $\Box A \in
\Gamma$, $\Box B \in \Delta$, and $C \in \mathcal{L}^{Ag(\Gamma)
\cup Ag(\Delta)}_{|\Gamma| \cap |\Delta |}$ such that both of the
following equations hold:
\begin{equation}\label{E:i1}
    \vdash (\Box A \wedge [j_1]A_1 \wedge\ldots\wedge [j_r]A_r) \to C,
\end{equation}
and
\begin{equation}\label{E:i2}
    \vdash (\Box B \wedge [i_1]B_1 \wedge\ldots\wedge [i_s]B_s) \to \neg
    C.
\end{equation}

Since $Ag(\Gamma)\cap Ag(\Delta) = \emptyset$, all of the agent
indices in the united sequence $j_1,\ldots, j_r, i_1,\ldots, i_s$
must be pairwise different and we must have $r + s \leq n$.
Therefore, $r + s \in \{ 0,1,2,3 \}$ which gives us our three
cases below. Although these cases show many similarities, we
consider them separately. In every case we reason by
contraposition, showing that the separability of \eqref{pair}
(expressed by \eqref{E:i1} and \eqref{E:i2}) implies the
separability of $(\Gamma, \Delta)$, thus contradicting the initial
assumption of the lemma.

\emph{Case 1}. Let $\{ r,s \} = \{ 1,2\}$. Assume, wlog, that $r =
2$ and $s = 1$, the other subcase is symmetric. Then, by
\eqref{E:i1} and \eqref{E:i2}, there exist $i,j$ and $k$ such that
$\{ i,j,k \} = \{ 1,2,3\}$, and that both of the following hold:
\begin{equation}\label{E:i4}
    \vdash (\Box A \wedge [i]A_1 \wedge [j]A_2) \to C,
\end{equation}
and
\begin{equation}\label{E:i5}
    \vdash (\Box B \wedge [k]B_1) \to \neg C.
\end{equation}
By Lemma \ref{technical}.2, \eqref{E:i4}, and propositional logic,
we get that:
\begin{equation}\label{E:i6}
    \vdash (\Box A \wedge \Diamond[i]A_1 \wedge \Diamond[j]A_2) \to \neg\Diamond[k]\neg C,
\end{equation}
On the other hand, by Lemma \ref{technical}.3 and \eqref{E:i5}:
\begin{equation}\label{E:i7}
    \vdash (\Box B \wedge \Diamond[k]B_1) \to \Diamond[k]\neg C.
\end{equation}
Since $C$, by its choice, is in $\mathcal{L}^{Ag(\Gamma) \cup
Ag(\Delta)}_{|\Gamma| \cap |\Delta |}$, we clearly have
$\Diamond[k]\neg C \in \mathcal{L}^{Ag(\Gamma) \cup
Ag(\Delta)}_{|\Gamma| \cap |\Delta |}$, and we also have, by the
initial choice of our formulas, that $\Box A, \Diamond[i]A_1,
\Diamond[j]A_2 \in \Gamma$ and  $\Box B, \Diamond[k]B_1 \in
\Delta$. Therefore, it follows from \eqref{E:i6} and \eqref{E:i7},
that $(\Gamma, \Delta)$ is separable.

\emph{Case 2}. Let $\{ r,s \} = \{ 1 \}$. Then, by \eqref{E:i1}
and \eqref{E:i2}, there exist $i,j \in \{ 1,2,3 \}$ such that $i
\neq j$ and both of the following hold:
\begin{equation}\label{E:i'4}
    \vdash (\Box A \wedge [i]A_1) \to C,
\end{equation}
and
\begin{equation}\label{E:i'5}
    \vdash (\Box B \wedge [j]B_1) \to \neg C.
\end{equation}
By Lemma \ref{technical}.2 and \eqref{E:i'4} we get that:
\begin{equation}\label{E:i'6}
    \vdash (\Box A \wedge \Diamond[i]A_1) \to \neg\Diamond[j]\neg C,
\end{equation}
On the other hand, by Lemma \ref{technical}.3 and \eqref{E:i'5}:
\begin{equation}\label{E:i'7}
    \vdash (\Box B \wedge \Diamond[j]B_1) \to \Diamond[j]\neg C.
\end{equation}
Since $C$, by its choice, is in $\mathcal{L}^{Ag(\Gamma) \cup
Ag(\Delta)}_{|\Gamma| \cap |\Delta |}$, we clearly have
$\Diamond[j]\neg C \in \mathcal{L}^{Ag(\Gamma) \cup
Ag(\Delta)}_{|\Gamma| \cap |\Delta |}$, and we also have, by the
initial choice of our formulas, that $\Box A, \Diamond[i]A_1 \in
\Gamma$ and $\Box B, \Diamond[j]B_1 \in \Delta$. Therefore, it
follows from \eqref{E:i'6} and \eqref{E:i'7}, that $(\Gamma,
\Delta)$ is again separable, contrary to our assumptions.

\emph{Case 3}. $0 \in \{ r,s \}$. We may assume, wlog, that $s =
0$, the other subcase being symmetric. By \eqref{E:i2}, we must
have then:
\begin{equation}\label{E:i9}
    \vdash \Box B \to \neg C.
\end{equation}
By S5 properties of $\Box$, we get then:
\begin{align}
     &\vdash (\Box A \wedge \Diamond[j_1]A_1 \wedge\ldots \wedge \Diamond[j_r]A_r) \to
     \Diamond C\label{E:i10}&&\text{(from \eqref{E:i1})}\\
 &\vdash \Box B \to \Box\neg C\label{E:i11}&&\text{(from \eqref{E:i9})}
\end{align}
It follows then, by the choice of the formulas involved, that
$(\Gamma, \Delta)$ is separable, contrary to our assumptions.

This exhausts the list of possible cases and thus the Lemma is
proved.
\end{proof}

\begin{lemma}\label{inseparable3}
Let $V$ be a set of propositional variables, $Ag$ a finite agent
community, and let $\Gamma, \Delta \subseteq \mathcal{L}^{Ag}_V$
be such that $(\Gamma,\Delta)$ is inseparable. Moreover, assume
that $\Gamma$ is $(Ag(\Gamma), |\Gamma|)$-maxiconsistent and
$\Delta$ is $(Ag(\Delta), |\Delta|)$-maxiconsistent. Then:
\begin{enumerate}
\item If $\neg\Box A_1 \in \Gamma$, then the pair $(\Gamma^\Box
\cup \{ \neg A_1\}, \Delta^\Box)$ is inseparable.

\item If $\neg\Box B_1 \in \Delta$, then the pair $(\Gamma^\Box,
\Delta^\Box \cup \{ \neg B_1\})$ is inseparable.

\item If $\neg[j]A_1 \in \Gamma$, then the pair $(\Gamma^\Box \cup
\Gamma^{[j]} \cup \{ \neg A_1\}, \Delta^\Box)$ is inseparable.

\item If $\neg[i]B_1 \in \Delta$, then the pair $(\Gamma^\Box,
\Delta^\Box  \cup \Delta^{[i]} \cup \{ \neg B_1\})$ is
inseparable.
\end{enumerate}
\end{lemma}
\begin{proof}
(Part 1).  Assume the hypothesis. If the pair $(\Gamma^\Box \cup
\{ \neg A_1\}, \Delta^\Box)$ is separable, then, by compactness of
stit logic, maxiconsistency of $\Gamma$ and $\Delta$, and S5
properties of all the modalities in stit logic, there must be
$\Box A \in \Gamma$, $\Box B \in \Delta$, and $C \in
\mathcal{L}^{Ag(\Gamma) \cup Ag(\Delta)}_{|\Gamma| \cap |\Delta
|}$ such that \eqref{E:i9} holds together with the following
equation:
\begin{equation}\label{E:j1}
    \vdash (\Box A \wedge \neg A_1) \to C.
\end{equation}
From \eqref{E:j1} we infer, using S5 properties of $\Box$:
\begin{equation}\label{E:j2}
    \vdash (\Box A \wedge \Diamond\neg A_1) \to \Diamond C.
\end{equation}
On the other hand, from \eqref{E:i9} we infer \eqref{E:i11}
arguing as in Case 3 in the proof of Lemma \ref{inseparable2}
above. Taken together, \eqref{E:i11} and \eqref{E:j2} show
separability of $(\Gamma, \Delta)$, contrary to our assumptions.
Therefore, \eqref{E:j1} and \eqref{E:i9} cannot hold, whence
$(\Gamma^\Box \cup \{ \neg A_1\}, \Delta^\Box)$  must be
inseparable, and we are done.

Part 2 is symmetric to Part 1.

(Part 3). Assume the hypothesis. If the pair $(\Gamma^\Box \cup
\Gamma^{[j]} \cup \{ \neg A_1\}, \Delta^\Box)$ is separable, then,
by compactness of stit logic, maxiconsistency of $\Gamma$ and
$\Delta$, and S5 properties of all the modalities in stit logic,
there must be $\Box A, [j]A' \in \Gamma$, $\Box B \in \Delta$, and
$C \in \mathcal{L}^{Ag(\Gamma) \cup Ag(\Delta)}_{|\Gamma| \cap
|\Delta |}$ such that \eqref{E:i9} holds together with the
following equation:
\begin{equation}\label{E:k1}
    \vdash (\Box A \wedge [j]A' \wedge \neg A_1) \to C
\end{equation}
Next we infer:
\begin{align}
&\vdash [j]((\Box A \wedge [j]A' \wedge \neg C) \to
A_1)\label{E:k3}&&\text{(by \eqref{E:k1}, $[j]$ is
S5)}\\
&\vdash ([j]\Box A \wedge [j]A' \wedge [j]\neg C) \to
[j]A_1\label{E:k3-1}&&\text{(by \eqref{E:k3}, $[j]$ is
S5)}\\
&\vdash \Box A \to [j]\Box A\label{E:k3-2}&&\text{(cf. \eqref{E:p3-6} above)}\\
&\vdash (\Box A \wedge [j]A' \wedge [j]\neg C) \to
[j]A_1\label{E:k3-3}&&\text{(by \eqref{E:k3-1} and
\eqref{E:k3-2})}\\
&\vdash (\Box A \wedge [j]A' \wedge \neg[j]A_1) \to \neg[j]\neg
C\label{E:k3-4}&&\text{(by \eqref{E:k3-3} and prop. logic)}
\end{align}
We also infer \eqref{E:i11} from \eqref{E:i9}, arguing as in Case
3 in the proof of Lemma \ref{inseparable2} above. From
\eqref{E:i11} and \eqref{A2} it then follows that:
\begin{equation}\label{E:k4}
    \vdash \Box B \to [j]\neg C
\end{equation}
Taken together, \eqref{E:k3-4} and \eqref{E:k4} show separability
of $(\Gamma, \Delta)$, contrary to our assumptions. Therefore,
\eqref{E:k1} and \eqref{E:i9} cannot hold together, whence
$(\Gamma^\Box \cup \Gamma^{[j]} \cup \{ \neg A_1\}, \Delta^\Box)$
must be inseparable, and we are done.

Part 4 is symmetric to Part 3.
\end{proof}

We are now prepared to prove Theorem \ref{positive}. Assume that
$n \leq 3$, assume for \emph{reductio}, that $A, B \in
\mathcal{L}^{\{ 1,\ldots, n \}}_{V}$, and we have $\vdash A \to
B$, $Ag(A) \cap Ag(B) = \emptyset$, but for no $C \in
\mathcal{L}^{Ag(A) \cup Ag(B)}_{|A| \cap |B|}$ we have both
$\vdash A \to C$ and $\vdash C \to B$. This means that the pair
$(\{ A \},\{ \neg B \})$ is inseparable and can be extended, using
Lemma \ref{inseparable1}, to an inseparable pair $(\Xi_0, \Xi_1)$
such that $\Xi_0$ is $(Ag(A), |A|)$-maxiconsistent and $\Xi_1$ is
$(Ag(B), |B|)$-maxiconsistent. We now define a $(Ag(A) \cup Ag(B),
|A| \cup |B|)$-stit model $\mathfrak{S}$ which we will show to
satisfy $\Xi_0 \cup \Xi_1$.

Now we start defining components of $\mathfrak{S} = \langle Tree,
\leq, Choice, V\rangle$:
\begin{itemize}
\item We first define the set of \emph{standard pairs} as the set
of all inseparable pairs $(\Gamma, \Delta)$ such that $\Gamma$ is
$(Ag(A), |A|)$-maxiconsistent, $\Delta$ is $(Ag(B),
|B|)$-maxiconsistent, and the following condition holds:
$$
\Xi^\Box_0 \subseteq \Gamma \& \Xi^\Box_1 \subseteq \Delta.
$$
The set of standard pairs is non-empty since $(\Xi_0, \Xi_1)$ is
clearly a standard pair.

\item We then define $Tree$ as the set of all standard pairs plus
 a single additional moment $\dag$.

\item $\leq$ is the reflexive closure of the relation $\{(\dag,
(\Gamma, \Delta))\mid(\Gamma, \Delta)\text{ is a standard pair }
\}$
\end{itemize}
Immediately we get the following lemma:
\begin{lemma}\label{boxed}
If $(\Gamma, \Delta)$ is a standard pair then $\Gamma^\Box =
\Xi^\Box_0$ and $\Delta^\Box = \Xi^\Box_1$.
\end{lemma}
\begin{proof}
We show that $\Gamma^\Box = \Xi^\Box_0$, the other part is
similar. We have $\Xi^\Box_0 \subseteq \Gamma$ by the definition
of standard pair, whence clearly $\Xi^\Box_0 \subseteq
\Gamma^\Box$. In the other direction, assume that $\Box C \in
\Gamma$. Since $\Xi_0$ is $(Ag(A), |A|)$-maxiconsistent, we must
have either $\Box C \in \Xi_0$ or $\neg\Box C \in \Xi_0$. In the
latter case, by S5 properties of $\Box$ and $(Ag(A),
|A|)$-maxiconsistency of $\Xi_0$ we get that $\Box\neg\Box C \in
\Xi_0$. We have established, therefore, that either $\Box C \in
\Xi^\Box_0$ or $\Box\neg\Box C \in \Xi^\Box_0$. However, we cannot
have $\Box\neg\Box C \in \Xi^\Box_0$, since we know that
$\Xi^\Box_0 \subseteq \Gamma$, and also $\Box C \in \Gamma$. It
follows that we must have $\Box C \in \Xi^\Box_0$.
\end{proof}
We pause to reflect on the structure of histories induced by the
pair $(Tree, \leq)$. Every such history has the form $h_{(\Gamma,
\Delta)} = \{\dag, (\Gamma, \Delta) \}$. It is clear, moreover,
that we have both $H^\mathfrak{S}_\dag = Hist(\mathfrak{S})$ and
$H_{(\Gamma, \Delta)} = \{ h_{(\Gamma, \Delta)}\}$ for every
standard pair $(\Gamma, \Delta)$. We then define the choice
function for our model in the following way:
\begin{itemize}
\item For every $j \in Ag(A)$ and standard pairs $(\Gamma,
\Delta)$ and $(\Gamma_0, \Delta_0)$, we define that $h_{(\Gamma_0,
\Delta_0)} \in Choice^\dag_j(h_{(\Gamma, \Delta)})$ iff
$\Gamma^{[j]} \subseteq \Gamma_0$.

\item Similarly, for every $i \in Ag(B)$ and standard pairs
$(\Gamma, \Delta)$ and $(\Gamma_0, \Delta_0)$, we define that
$h_{(\Gamma_0, \Delta_0)} \in Choice^\dag_j(h_{(\Gamma, \Delta)})$
iff $\Delta^{[i]} \subseteq \Delta_0$.

\item For every $j \in Ag(A) \cup Ag(B)$ and every standard pair
$(\Gamma, \Delta)$ we set that $Choice^{(\Gamma, \Delta)}_j = \{
H_{(\Gamma, \Delta)}\} = \{\{ h_{(\Gamma, \Delta)}\}\}$.

\item Finally, for a $p \in |A|$, we define that $V(p) = \{ (\dag,
(\Gamma, \Delta))\mid p \in \Gamma \}$; symmetrically, for a $q
\in |B|$, we define that $V(q) = \{ (\dag, (\Gamma, \Delta))\mid q
\in \Delta \}$.
\end{itemize}

First of all, we need to show that we have in fact defined a stit
model:
\begin{lemma}\label{model}
The structure $\mathfrak{S} = \langle Tree, \leq, Choice,
V\rangle$, as defined above, is a $(Ag(A) \cup Ag(B), |A| \cup
|B|)$-stit model.
\end{lemma}
\begin{proof} It is obvious that $\leq$ is a forward-branching
preorder on the non-empty set $Tree$. The fact that $Choice^m_j$
is a partition of $H^\mathfrak{S}_m$ trivially follows from
definition, whenever $m \neq \dag$. If, on the other hand, $m =
\dag$, then this same fact follows from S5 properties of $[j]$
together with the fact that, for every standard pair $(\Gamma,
\Delta)$, $Ag(\Gamma) \cap Ag(\Delta) = Ag(A) \cap Ag(B) =
\emptyset$.

As for the constraints, \eqref{HC} is satisfied since $\dag$ is
the $\leq$-least moment in $Tree$ and \eqref{NCUH} is satisfied
because there are no undivided histories in $\mathfrak{S}$. We
consider \eqref{IA}. Let $m \in Tree$ and let $f$ be a function on
$Ag$ such that $(\forall j \in Ag(A) \cup Ag(B))(f(j) \in
Choice^m_j)$. We are going to show that in this case $\bigcap_{j
\in Ag(A) \cup Ag(B)}f(j) \neq \emptyset$. If $m \neq \dag$, then
this is obvious, since every agent will have a vacuous choice. We
treat the case when $m = \dag$.

Then, for every $j \in Ag(A) \cup
Ag(B)$, we pick an $h_j \in f(j)$ so that $f(j) =
Choice^\dag_j(h_j)$. Since $Hist(Tree, \leq) = \{h_{(\Gamma,
\Delta)}\mid (\Gamma, \Delta)\text{ is a standard pair}\}$, we can
choose, for every $j \in Ag(A) \cup Ag(B)$, a standard pair
$(\Gamma_j, \Delta_j)$ such that $h_j = h_{(\Gamma_j, \Delta_j)}$.
Together with $f(j) = Choice^\dag_j(h_j)$, this gives us the
following equation:
\begin{equation}\label{E:equ}
    (\forall j \in Ag(A) \cup Ag(B))(f(j) = Choice^\dag_j(h_{(\Gamma_j, \Delta_j)})
\end{equation}
Now consider the pair:
\begin{equation}\label{E:ind-pair1}
    (\Xi^\Box_0 \cup \bigcup\{ \Gamma^{[j]}_j \mid j \in Ag(A)\}, \Xi^\Box_1 \cup \bigcup\{ \Delta^{[i]}_i \mid i \in Ag(B)\}
\end{equation}
We will show that the pair \eqref{E:ind-pair1} is inseparable.
Indeed, suppose otherwise. Then, by Lemma \ref{L:separability},
there must be $\Box A_1,\ldots,\Box A_r \in \Xi^\Box_0$, $\Box
B_1,\ldots,\Box B_{r'} \in \Xi^\Box_1$,
$[j]A^j_1,\ldots,[j]A^j_{r(j)} \in \Gamma_j$ (for every $j \in
Ag(A)$), and $[i]B^i_1,\ldots,[i]B^i_{r(i)} \in \Delta_i$ (for
every $i \in Ag(B)$) such that the pair:
\begin{align}
    (\{ \Box A_1,\ldots,\Box A_r \} \cup \bigcup&\{ \{ [j]A^j_1,\ldots,[j]A^j_{r(j)} \} \mid j \in
    Ag(A)\},\notag\\
     &\{\Box
B_1,\ldots,\Box B_{r'}\}\cup \bigcup\{ \{
[i]B^i_1,\ldots,[i]B^i_{r(i)})\} \mid i \in
Ag(B)\})\label{E:ind-pair2}
\end{align}
is separable. Now the contraposition of Lemma \ref{inseparable1}.2
entails that in this case also the pair:
\begin{align}
    (\Xi^\Box_0 \cup \bigcup&\{ \{ [j]A^j_1,\ldots,[j]A^j_{r(j)} \} \mid j \in
    Ag(A)\},\notag\\
     &\Xi^\Box_1\cup \bigcup\{ \{
[i]B^i_1,\ldots,[i]B^i_{r(i)}\} \mid i \in
Ag(B)\})\label{E:ind-pair3}
\end{align}
must be separable. Next, for every $j \in Ag(A)$ and every $i \in
Ag(B)$, we set:
$$
\alpha_j := A^j_1\wedge\ldots\wedge
A^j_{r(j)};\qquad\qquad\qquad\qquad \beta_i :=
B^i_1\wedge\ldots\wedge B^i_{r(i)}.
$$
By Lemma \ref{technical}.1 and the separability of the pair
\eqref{E:ind-pair3}, we know that also the following pair must be
separable:
\begin{equation}\label{E:ind-pair4}
    (\Xi^\Box_0 \cup \{ [j]\alpha_j \mid j \in Ag(A)\}, \Xi^\Box_1 \cup \{
[i]\beta_i \mid i \in Ag(B)\}.
\end{equation}
For every $j \in Ag(A)$, the formulas
$[j]A^j_1,\ldots,[j]A^j_{r(j)}$ were chosen in $\Gamma_j$,
therefore, it follows from Lemma \ref{technical}.1 and
maxiconsistency of $\Gamma_j$ that also $[j]\alpha_j \in
\Gamma_j$. By S5 properties of $\Box$, this means that also
$\Diamond[j]\alpha_j \in \Gamma_j$ so that, by consistency,
$\Box\neg[j]\alpha_j \notin \Gamma_j$. The latter means, by Lemma
\ref{boxed}, that $\Box\neg[j]\alpha_j \notin \Xi_0$, therefore,
by maxiconsistency, $\Diamond[j]\alpha_j \in \Xi_0$. By a parallel
argument, one can also show that, for every $i \in Ag(B)$,
$\Diamond[i]\beta_i \in \Xi_1$. Therefore, by Lemma
\ref{inseparable2}, the separability of the pair
\eqref{E:ind-pair4} entails the separability of $(\Xi_0, \Xi_1)$
which contradicts the choice of the latter pair. The obtained
contradiction shows that the pair \eqref{E:ind-pair1} must be
inseparable.

Therefore, by Lemma \ref{inseparable1}.1, the pair
\eqref{E:ind-pair1} can be extended to a pair $(\Gamma_0,
\Delta_0)$ such that $\Gamma_0$ is $(Ag(A), |A|)$-maxiconsistent
and $\Delta_0$ is $(Ag(B), |B|)$-maxiconsistent. By the choice of
\eqref{E:ind-pair1}, it is also clear that both $\Xi^\Box_0
\subseteq \Gamma_0$ and $\Xi^\Box_1 \subseteq \Delta_0$, which
means that $(\Gamma_0, \Delta_0)$ is a standard pair. Therefore,
we must have $h_{(\Gamma_0, \Delta_0)} \in H^\mathfrak{S}_\dag$.
Now, let $j \in Ag(A)$. Then, by the choice of
\eqref{E:ind-pair1}, $\Gamma^{[j]}_j \subseteq \Gamma_0$, whence
we get, by \eqref{E:equ} and the definition of $Choice$, that
$h_{(\Gamma_0, \Delta_0)} \in Choice^\dag_j(h_{(\Gamma_j,
\Delta_j)}) = f(j)$. Similarly, if $i \in Ag(B)$, then, by the
choice of \eqref{E:ind-pair1}, $\Delta^{[i]}_i \subseteq
\Delta_0$, whence we get, by \eqref{E:equ} and the definition of
$Choice$, that $h_{(\Gamma_0, \Delta_0)} \in
Choice^\dag_i(h_{(\Gamma_i, \Delta_i)}) = f(i)$. Summing up, we
obtain that:
$$
h_{(\Gamma_0, \Delta_0)} \in \bigcap_{j \in Ag(A) \cup Ag(B)}f(j)
\neq \emptyset,
$$
and \eqref{IA} is thus satisfied.
\end{proof}

For the defined model $\mathfrak{S}$, we show the following truth
lemma:
\begin{lemma}\label{truth}
Let $\mathfrak{S}$ be as defined above, let $(\Gamma, \Delta)$ be
a standard pair, let $C \in \mathcal{L}^{Ag(A)}_{|\Gamma|}$, and
let $D \in \mathcal{L}^{Ag(B)}_{|\Delta|}$. Then:
\begin{enumerate}
\item $\mathfrak{S}, \dag, h_{(\Gamma, \Delta)} \models C
\Leftrightarrow C \in \Gamma$;

\item $\mathfrak{S}, \dag, h_{(\Gamma, \Delta)} \models D
\Leftrightarrow D \in \Delta$.
\end{enumerate}
\end{lemma}
\begin{proof}
We show Part 1, the other part is similar. The proof proceeds by
induction on the construction of $C$.

\emph{Basis}. $C = p \in |\Gamma|$. Then:
$$
\mathfrak{S}, \dag, h_{(\Gamma, \Delta)} \models p \Leftrightarrow
(\dag, h_{(\Gamma, \Delta)}) \in V(p) \Leftrightarrow p \in
\Gamma,
$$
by the definition of $V$ above.

\emph{Induction step}. The Boolean cases are strightforward. We
treat the modal cases:

\emph{Case 1}. $C = \Box D$. ($\Leftarrow$) Assume that $\Box D
\in \Gamma$ and take an arbitrary $g \in H^\mathfrak{S}_\dag$. We
will show that $\mathfrak{S}, \dag, g  \models D$. Indeed, we must
have $g = h_{(\Gamma_0, \Delta_0)}$ for an appropriate standard
pair $(\Gamma_0, \Delta_0)$. By Lemma \ref{boxed}, we must have
$\Gamma^\Box = \Xi^\Box_0 = \Gamma^\Box_0$, whence it follows that
$\Box D \in \Gamma_0$. By S5 properties of $\Box$ and $(Ag(A),
|A|)$-maxiconsistency of $\Gamma_0$, it follows further that $D
\in \Gamma_0$, whence $\mathfrak{S}, \dag, g (= h_{(\Gamma_0,
\Delta_0)}) \models D$ by induction hypothesis. Since $g$ was
chosen in $H^\mathfrak{S}_\dag$ arbitrarily, it follows that
$\mathfrak{S}, \dag, h_{(\Gamma, \Delta)} \models \Box D$.

($\Rightarrow$). Assume that $\Box D \notin \Gamma$. By $(Ag(A),
|A|)$-maxiconsistency of $\Gamma$, we must have then that
$\neg\Box D \in \Gamma$, which, by Lemma \ref{inseparable3}.1,
means that the pair $(\Gamma^\Box \cup \{ \neg D\}, \Delta^\Box)$
must be inseparable. By Lemma \ref{boxed}, we know that also the
pair $(\Xi_0^\Box \cup \{ \neg D\}, \Xi_1^\Box)$ must be
inseparable. We then extend the latter pair, using Lemma
\ref{inseparable1}.1, to a standard pair $(\Gamma_0, \Delta_0)$.
It is clear that $D \notin \Gamma_0$, hence, by induction
hypothesis, $\mathfrak{S}, \dag, h_{(\Gamma_0, \Delta_0)}
\not\models D$. Since $h_{(\Gamma_0, \Delta_0)} \in
H^\mathfrak{S}_\dag$, this further means that $\mathfrak{S}, \dag,
h_{(\Gamma, \Delta)} \not\models \Box D$, as desired.

\emph{Case 2}. $C = [j]D$ for some $j \in Ag(A)$. ($\Leftarrow$)
Assume that $[j]D \in \Gamma$ and take an arbitrary $g \in
Choice^\dag_j(h_{(\Gamma, \Delta)})$. We will show that
$\mathfrak{S}, \dag, g  \models D$. Indeed, we must have $g =
h_{(\Gamma_0, \Delta_0)}$ for an appropriate standard pair
$(\Gamma_0, \Delta_0)$. Given that $h_{(\Gamma_0, \Delta_0)} = g
\in Choice^\dag_j(h_{(\Gamma, \Delta)})$, we must also have, by
the definition of $Choice$, that $\Gamma^{[j]}\subseteq \Gamma_0$.
Therefore, $[j]D \in \Gamma_0$, and it follows by S5 properties of
$[j]$ and $(Ag(A), |A|)$-maxiconsistency of $\Gamma_0$, that also
$D \in \Gamma_0$ whence $\mathfrak{S}, \dag, g (= h_{(\Gamma_0,
\Delta_0)}) \models D$ by the induction hypothesis. Since $g$ was
chosen in $Choice^\dag_j(h_{(\Gamma, \Delta)})$ arbitrarily, we
have shown that $\mathfrak{S}, \dag, h_{(\Gamma, \Delta)} \models
[j]D$.

($\Rightarrow$). Assume that $[j]D \notin \Gamma$. By $(Ag(A),
|A|)$-maxiconsistency of $\Gamma$, we must have then that
$\neg[j]D \in \Gamma$, which, by Lemma \ref{inseparable3}.3, means
that the pair $(\Gamma^\Box \cup \Gamma^{[j]}\cup \{ \neg D\},
\Delta^\Box)$ must be inseparable. By Lemma \ref{boxed}, we know
that also the pair $(\Xi_0^\Box \cup \Gamma^{[j]} \cup \{ \neg
D\}, \Xi_1^\Box)$ must be inseparable. We then extend the latter
pair, using Lemma \ref{inseparable1}.1, to a standard pair
$(\Gamma_0, \Delta_0)$. It is clear that $D \notin \Gamma_0$,
hence, by induction hypothesis, $\mathfrak{S}, \dag, h_{(\Gamma_0,
\Delta_0)} \not\models D$. We also clearly have
$\Gamma^{[j]}\subseteq \Gamma_0$, which means that $h_{(\Gamma_0,
\Delta_0)} \in Choice^\dag_j(h_{(\Gamma, \Delta)})$. Therefore, we
get that $\mathfrak{S}, \dag, h_{(\Gamma, \Delta)} \not\models
[j]D$, as desired.
\end{proof}

We can now finish our proof of Theorem \ref{positive} by recalling
the fact that we have, according to the above assumption, both $A
\in \Xi_0$ and $\neg B \in \Xi_1$, so that it follows from Lemma
\ref{truth}, that:
$$
\mathfrak{S}, \dag, h_{(\Xi_0, \Xi_1)} \models A \wedge \neg B.
$$
The latter is in contradiction with the assumption that $\vdash A
\to B$, and this contradiction means that there must be an
interpolant for this implication.

\section{The case $n > 3$}\label{S:negative}
The main result of this section looks as follows:
\begin{theorem}\label{negative}
For every $n > 3$, stit logic does not have $(RCIP)_n$.
\end{theorem}

Again, we start with some technicalities:

\begin{lemma}\label{counterexample}
Let $j_1, j_2, j_3, j_4 \in Ag$ and propositional variables $p, q,
r$ be pairwise different. Then:
$$
\vdash \Diamond([j_1]p \wedge [j_2](p \to q)) \to
\neg\Diamond([j_3]r \wedge [j_4](r \to \neg q)).
$$
\end{lemma}
\begin{proof}
We reason as follows:
\begin{align}
&\Diamond([j_1]p \wedge [j_2](p \to q)) \wedge\Diamond([j_3]r
\wedge [j_4](r \to \neg q))\label{E:ce1}&&\text{(premise)}\\
&\Diamond([j_1]p \wedge [j_2](p \to q)) \to (\Diamond[j_1]p \wedge
\Diamond[j_2](p \to q))\label{E:ce2}&&\text{($\Box$ is S5)}\\
&\Diamond([j_3]r \wedge [j_4](r \to \neg q)) \to (\Diamond[j_3]r
\wedge
\Diamond[j_4](r \to \neg q))\label{E:ce3}&&\text{($\Box$ is S5)}\\
&\Diamond[j_1]p \wedge \Diamond[j_2](p \to q) \wedge\Diamond[j_3]r
\wedge \Diamond[j_4](r \to \neg q)\label{E:ce4}&&\text{(from \eqref{E:ce1}--\eqref{E:ce3})}\\
&\Diamond([j_1]p \wedge [j_2](p \to q) \wedge[j_3]r
\wedge [j_4](r \to \neg q))\label{E:ce5}&&\text{(from \eqref{E:ce4}, \eqref{A3})}\\
&([j_1]p \wedge [j_2](p \to q) \wedge[j_3]r \wedge [j_4](r \to
\neg q)) \to\notag\\
&\qquad\qquad\qquad\qquad\to (p \wedge (p \to q) \wedge r \wedge (r \to \neg q))\label{E:ce6}&&\text{($[j_1]$--$[j_4]$ are S5)}\\
&([j_1]p \wedge [j_2](p \to q) \wedge[j_3]r \wedge [j_4](r \to
\neg q)) \to\bot\label{E:ce7}&&\text{(from \eqref{E:ce6} by prop. logic)}\\
&\Diamond([j_1]p \wedge [j_2](p \to q) \wedge[j_3]r \wedge [j_4](r
\to
\neg q)) \to\bot\label{E:ce8}&&\text{(from \eqref{E:ce7} since $\Box$ is S5)}\\
&\bot\label{E:ce9}&&\text{(from \eqref{E:ce5} and \eqref{E:ce8})}
\end{align}
\end{proof}

\begin{definition}\label{bisimulation}
Let $\mathfrak{S} = \langle Tree, \leq, Choice, V\rangle$ and
$\mathfrak{S}' = \langle Tree', \leq', Choice', V'\rangle$ be
$(Ag, V)$-stit models, and let $m \in Tree$ and $m' \in Tree'$.
Relation $B \in H^\mathfrak{S}_m \times H^{\mathfrak{S}'}_{m'}$ we
will call a \emph{bisimulation between} $(\mathfrak{S}, m)$
\emph{and} $(\mathfrak{S}', m')$, iff the domain of $B$ is
$H^\mathfrak{S}_m$, the counter-domain of $B$ is
$H^{\mathfrak{S}'}_{m'}$, and the following holds for all $p \in
V$, all $j \in Ag$, all $h_1, h_2 \in H^\mathfrak{S}_m$ and all
$h'_1, h'_2 \in H^{\mathfrak{S}'}_{m'}$:
\begin{align}
& h_1\mathrel{B}h'_1 \Rightarrow (\mathfrak{S}, m, h_1 \models p
\Leftrightarrow \mathfrak{S}', m', h'_1 \models p)\label{atoms}\tag{\text{atoms}}\\
&(h_1\mathrel{B}h'_1 \& h_2 \in Choice^m_j(h_1)) \Rightarrow
(\exists h'_3 \in (Choice')^{m'}_j(h'_1))(h_2\mathrel{B}h'_3)\label{forth}\tag{\text{forth}}\\
&(h_1\mathrel{B}h'_1 \& h'_2 \in (Choice')^{m'}_j(h'_1))
\Rightarrow (\exists h_3 \in
Choice^m_j(h_1))(h_3\mathrel{B}h'_2)\label{back}\tag{\text{back}}
\end{align}
\end{definition}

We show that existence of a bisimulation implies the equality of
theories:
\begin{lemma}\label{L:bisimulation}
Let $\mathfrak{S} = \langle Tree, \leq, Choice, V\rangle$ and
$\mathfrak{S}' = \langle Tree', \leq', Choice', V'\rangle$ be
$(Ag, V)$-stit models, and let $B \in H^\mathfrak{S}_m \times
H^{\mathfrak{S}'}_{m'}$ be a bisimulation between $(\mathfrak{S},
m)$ and $(\mathfrak{S}', m')$. Then, for all $A \in
\mathcal{L}^{Ag}_V$ and all $h_1 \in H^\mathfrak{S}_m$ and $h'_1
\in H^{\mathfrak{S}'}_{m'}$:
$$
h_1\mathrel{B}h'_1 \Rightarrow (\mathfrak{S}, m, h_1 \models A
\Leftrightarrow \mathfrak{S}', m', h'_1 \models A).
$$
\end{lemma}
\begin{proof} By induction on the construction of $A$. The basis
follows from \eqref{atoms}, and the Boolean cases in the induction
step are trivial. We consider the modal cases:

\emph{Case 1}. $A$ has the form $\Box B$. ($\Rightarrow$) Assume
that $\mathfrak{S}, m, h_1 \models \Box B$ and let $h'_2 \in
H^{\mathfrak{S}'}_{m'}$ be arbitrary. Then, since the
counter-domain of $B$ is $H^{\mathfrak{S}'}_{m'}$, choose any $h_2
\in H^\mathfrak{S}_m$ such that $h_2\mathrel{B}h'_2$. We have
$\mathfrak{S}, m, h_2 \models B$, whence, by induction hypothesis,
it follows that $\mathfrak{S}', m', h'_2 \models B$. Since $h'_2
\in H^{\mathfrak{S}'}_{m'}$ was chosen arbitrarily, we infer that
$\mathfrak{S}', m', h'_1 \models \Box B = A$. ($\Leftarrow$)
Similarly to the ($\Rightarrow$)-part, using this time the fact
that the domain of $B$ is $H^\mathfrak{S}_m$.

\emph{Case 2}. $A$ has the form $[j]B$ for some $j \in Ag$.
($\Rightarrow$) Assume that $\mathfrak{S}, m, h_1 \models [j]B$
and let $h'_2 \in Choice'^{m'}_j(h'_1)$ be arbitrary. Using
condition \eqref{back}, choose a $h_3 \in Choice^{m}_j(h_1)$ such
that $h_3\mathrel{B}h'_2$. We have $\mathfrak{S}, m, h_3 \models
B$, whence, by induction hypothesis, it follows that
$\mathfrak{S}', m', h'_2 \models B$. Since $h'_2 \in
Choice'^{m'}_j(h'_1)$ was chosen arbitrarily, we infer that
$\mathfrak{S}', m', h'_1 \models [j]B = A$. ($\Leftarrow$)
Similarly to the ($\Rightarrow$)-part, using this time condition
\eqref{forth} instead of \eqref{back}.
\end{proof}

Now we need to define two models: a $(\{1,2,3,4\}, \{ p,
q\})$-stit model $\mathfrak{S} = \langle Tree, \leq, Choice,
V\rangle$, and a $(\{1,2,3,4\}, \{ q, r\})$-stit model
$\mathfrak{S}' = \langle Tree', \leq', Choice', V'\rangle$ to be
used in the proof of Theorem \ref{negative}. First, we define one
auxiliary set:
$$
4Tup := \{ (a,b,c,d)^+, (a,b,c,d)^- \mid a,b,c,d \in \{0,1\} \}.
$$
Next, we start with the definitions of the models, beginning with
their temporal substructures.
\begin{definition}\label{D:temporal}
We set:
\begin{enumerate}
\item $Tree := \{ \dag\} \cup 4Tup$.

\item $\leq$ is the reflexive closure of $\{ (\dag,m) \mid m \in
4Tup\}$.

\item $Tree' := \{ \ddag\} \cup 4Tup$.

\item $\leq'$ is the reflexive closure of $\{ (\ddag,m) \mid m \in
4Tup\}$.
\end{enumerate}
\end{definition}
For an integer $1 \leq j \leq 4$, by the $j$-th projection of $m
\in 4Tup = Tree \cap Tree'$ we will mean the $j$-th projection of
the corresponding $4$-tuple, regardless of whether $m$ is signed
by $+$ or $-$. Thus, for any appropriate $a,b,c,d \in \{0,1\}$,
the two elements $(a,b,c,d)^+$ and $(a,b,c,d)^-$ have the same
$j$-th projection for every $1 \leq j \leq 4$. For an $m \in 4Tup$
and an integer $1 \leq j \leq 4$, the $j$-th projection of $m$
will be denoted by $pr_j(m)$. The element from $\{ +, - \}$ by
which $m$ is signed, we will denote $sign(m)$ so that, e.g.,
$sign((a,b,c,d)^+) = +$. Finally, the complete $4$-tuple signed by
$sign(m)$ will be called the \emph{core} of $m$ and will be
denoted by $core(m)$ so that $core(m) = (pr_1(m), pr_2(m),
pr_3(m), pr_4(m))$.

The history structure induced by these definitions is as follows.
For $\mathfrak{S}$ we get that:
\begin{equation}\label{E:hists}
Hist(\mathfrak{S}) = \{ h_m = (\dag, m)\mid m \in 4Tup \} =
H^\mathfrak{S}_\dag
\end{equation}
Similarly, for $\mathfrak{S}'$ we get that:
\begin{equation}\label{E:hists'}
Hist(\mathfrak{S}') = \{ g_m = (\ddag, m)\mid m \in 4Tup \} =
H^{\mathfrak{S}'}_\ddag
\end{equation}
Once we know the sets of histories induced by $\mathfrak{S}$ and
$\mathfrak{S}'$, respectively, it is immediate to deduce the fans
of histories passing through any given moment in these models.
Namely, it follows that:
\begin{equation}\label{E:hists-dag}
    H^\mathfrak{S}_\dag = Hist(\mathfrak{S}),\qquad H^\mathfrak{S}_m = \{ h_m
    \},\qquad\text{for all }m \in 4Tup
\end{equation}
and:
\begin{equation}\label{E:hists-ddag}
    H^{\mathfrak{S}'}_\ddag = Hist(\mathfrak{S}'),\qquad H^{\mathfrak{S}'}_m = \{ g_m
    \},\qquad\text{for all }m \in 4Tup
\end{equation}

 This insight into the history structure allows for a handy
definition of choice functions and variable evaluations for the
two models:
\begin{definition}\label{D:choice}
We set that:
\begin{enumerate}
\item $Choice^\dag_j = \{ \{ h_{m} \mid pr_j(m) = 0 \}, \{ h_{m}
\mid pr_j(m) = 1 \}\}$ for all $1 \leq j \leq 4$.

\item $Choice^m_j = \{ H^\mathfrak{S}_m \} = \{ \{ h_m \}\}$ for
all $m \in 4Tup$ and $1 \leq j \leq 4$.

\item $V(p) = \{ (\dag, h_m) \mid pr_1(m) = 0 \}$,

$V(q) = \{ (\dag, h_m) \mid (pr_1(m) = pr_2(m) = 0) \vee (pr_3(m)
= pr_4(m) = 0)\vee sign(m) = + \}$.

\item $Choice'^\ddag_j = \{ \{ g_{m} \mid pr_j(m) = 0 \}, \{ g_{m}
\mid pr_j(m) = 1 \}\}$ for all $1 \leq j \leq 4$.

\item $Choice'^m_j = \{ H^{\mathfrak{S}'}_m \} = \{ \{ g_m \}\}$
for all $m \in 4Tup$ and $1 \leq j \leq 4$.

\item $V'(q) = \{ (\ddag, g_m) \mid (pr_3(m) = pr_4(m) = 0) \vee
(sign(m) = + \& (pr_3(m) \neq 1 \vee pr_4(m) \neq 0)) \}$,

$V'(r) = \{ (\ddag, g_m) \mid pr_3(m)  = 1 \}$.
\end{enumerate}
\end{definition}
We now establish a number of further lemmas and corollaries.
\begin{corollary}\label{L:choice-cor}
Let $1 \leq j \leq 4$. Then $Choice^\dag_j(h_m) = \{ h_{m_1} \mid
pr_j(m) = pr_j(m_1) \}$ and $Choice'^\ddag_j(g_m) = \{ g_{m_1}
\mid pr_j(m) = pr_j(m_1) \}$ for all $m \in 4Tup$.
\end{corollary}
\begin{proof}
The Corollary follows immediately from Definition \ref{D:choice}.1
and \ref{D:choice}.4, and the fact that for every $m \in 4Tup$ we
have either $pr_j(m) = 0$ or $pr_j(m) = 1$.
\end{proof}
\begin{lemma}\label{L:modelhood}
$\mathfrak{S}$, as given in Definitions \ref{D:temporal} and
\ref{D:choice}, is a $(\{1,2,3,4\}, \{ p, q\})$-stit model,
whereas $\mathfrak{S}'$, as given in the same Definitions, is a
$(\{1,2,3,4\}, \{ q, r \})$-stit model.
\end{lemma}
\begin{proof}
We consider $\mathfrak{S}$ first. Indeed, $\leq$ is obviously a
forward-branching partial order and $\dag$ is the $\leq$-least
element in $Tree$ so that \eqref{HC} is satisfied. Also, there are
no undivided histories at any moment of $Tree$ so that
\eqref{NCUH} is also satisfied trivially. Next, for any $m \in
4Tup$ and $1 \leq j \leq 4$, $Choice^m_j$ is a trivial partition
of $H^\mathfrak{S}_m$. As for $\dag$ itself, we have, by
Definition \ref{D:choice}.1, that, for any $1 \leq j \leq 4$,
$Choice^\dag_j = \{ \{ h_{m} \mid pr_j(m) = 0 \}, \{ h_{m} \mid
pr_j(m) = 1 \}\}$, which is obviously a pair of disjoint subsets
of $H^\mathfrak{S}_\dag = Hist(\mathfrak{S})$ such that their
union makes up $H^\mathfrak{S}_\dag = Hist(\mathfrak{S})$ itself.
The non-emptiness of both sets in this pair follows from the fact
that $(0,0,0,0)^+$ and $(1,1,1,1)^+$ are in $4Tup$. Finally, we
tackle \eqref{IA}. Assume that $f$ is defined on $\{ 1,2,3,4 \}$
in such a way that, for a given $m \in Tree$, we have $f(j) \in
Choice^m_j$ for all $1 \leq j \leq 4$. If $m \neq \dag$, then
clearly $\bigcap_{1 \leq j \leq 4}f(j) = H^\mathfrak{S}_m \neq
\emptyset$. On the other hand, if $m = \dag$, then, for every $1
\leq j \leq 4$, choose an $h_j \in f(j)$ so that we get $f(j) =
Choice^\dag_j(h_{j})$ for all $1 \leq j \leq 4$. Then it follows
from \eqref{E:hists} that, for every $1 \leq j \leq 4$, there must
exist an $m_j \in 4Tup$ such that $h_j = h_{m_j}$. But then,
consider the $4$-tuple $m_0 = (pr_1(m_1), pr_2(m_2), pr_3(m_3),
pr_4(m_4))^+$. It is immediate from Definition \ref{D:choice}.1
and Corollary \ref{L:choice-cor} that for every $1 \leq j \leq 4$
we have $h_{m_0} \in Choice^\dag_j(h_{m_j}) = f(j)$ whence
$h_{m_0} \in \bigcap_{1 \leq j \leq 4}f(j) \neq \emptyset$.

The proof of the Lemma for $\mathfrak{S}'$ is similar.
\end{proof}

\begin{lemma}\label{satisfied}
We have both:
$$
\mathfrak{S}, \dag, h_m \models \Diamond([1]p \wedge [2](p \to
q)),
$$
and:
$$
\mathfrak{S}', \ddag, g_m \models\Diamond([3]r \wedge [4](r \to
\neg q)),
$$
for all $m \in 4Tup$.
\end{lemma}
\begin{proof}
As for the first part of the Lemma, let $m := (0,0,0,0)^+$ and
consider $h_m$. If $h \in Choice^\dag_1(h_m)$ is chosen
arbitrarily, then, by \eqref{E:hists-dag}, $h = h_{m_1}$ for some
$m_1 \in 4Tup$ and, moreover, $pr_1(m_1) = pr_1(m) = 0$. But then,
by Definition \ref{D:choice}.3, $(\dag, h_{m_1}) \in V(p)$ so that
$\mathfrak{S}, \dag, h_{m_1} \models p$. Since $h_{m_1} \in
Choice^\dag_1(h_m)$ was arbitrary, this means that $\mathfrak{S},
\dag, h_{m} \models [1]p$.

Furthermore, let $h \in Choice^\dag_2(h_m)$ be chosen arbitrarily.
Then, again by \eqref{E:hists-dag}, $h = h_{m_1}$ for some $m_1
\in 4Tup$ and, moreover, $pr_2(m_1) = pr_2(m) = 0$. If
$\mathfrak{S}, \dag, h_{m_1} \models p$, this means that $(\dag,
h_{m_1}) \in V(p)$ so that also $pr_1(m_1) = 0$. But in this case
we will have $pr_1(m_1) = pr_2(m_1) = 0$ which means that also
$\mathfrak{S}, \dag, h_{m_1} \models q$. Thus we have shown, for
an arbitrary $h_{m_1} \in Choice^\dag_2(h_m)$, that whenever
$\mathfrak{S}, \dag, h_{m_1} \models p$, it is also the case that
$\mathfrak{S}, \dag, h_{m_1} \models q$ whence it follows that
$\mathfrak{S}, \dag, h_{m} \models [2](p \to q)$.

Summing up, we must have $\mathfrak{S}, \dag, h_{m} \models [1]p
\wedge [2](p \to q)$ for $m = (0,0,0,0)^+$, whence, given the
semantics of $\Box$ and \eqref{E:hists-dag}, it follows that
$\mathfrak{S}, \dag, h_m \models \Diamond([1]p \wedge [2](p \to
q))$ for all $m \in 4Tup$.

Turning now to the second part of the Lemma, we set $m :=
(0,0,1,0)^+$ and consider $g_m$. If $g \in Choice'^\ddag_3(g_m)$
is chosen arbitrarily, then, by \eqref{E:hists-dag}, $g = g_{m_1}$
for some $m_1 \in 4Tup$ and, moreover, $pr_3(m_1) = pr_3(m) = 1$.
But then, by Definition \ref{D:choice}.6, $(\ddag, g_{m_1}) \in
V'(r)$ so that $\mathfrak{S}', \ddag, g_{m_1} \models r$. Since
$g_{m_1} \in Choice'^\ddag_3(g_m)$ was arbitrary, this means that
$\mathfrak{S}', \ddag, g_{m} \models [3]r$.

Furthermore, let $g \in Choice'^\ddag_4(g_m)$ be chosen
arbitrarily. Then, again by \eqref{E:hists-dag}, $g = g_{m_1}$ for
some $m_1 \in 4Tup$ and, moreover, $pr_4(m_1) = pr_4(m) = 0$. If
$\mathfrak{S}', \ddag, g_{m_1} \models r$, this means that
$(\ddag, g_{m_1}) \in V'(r)$ so that also $pr_3(m_1) = 1$. But in
this case we will have both $pr_3(m_1) = 1$ and $pr_4(m_1) = 0$
which means that also $\mathfrak{S}', \ddag, g_{m_1} \models \neg
q$. Thus we have shown, for an arbitrary $g_{m_1} \in
Choice'^\ddag_4(g_m)$, that whenever $\mathfrak{S}', \ddag,
g_{m_1} \models r$, it is also the case that $\mathfrak{S}',
\ddag, g_{m_1} \models \neg q$ whence it follows that
$\mathfrak{S}', \ddag, g_{m} \models [4](r \to \neg q)$.

Summing up, we must have $\mathfrak{S}', \ddag, g_{m} \models [3]r
\wedge [4](r \to \neg q)$ for $m = (0,0,1,0)^+$, which means,
given the semantics of $\Box$ and \eqref{E:hists-ddag}, that
$\mathfrak{S}', \ddag, g_m \models \Diamond([3]r \wedge [4](r \to
\neg q))$ for all $m \in 4Tup$.
\end{proof}
In what follows we let $\mathfrak{S}_q$ and $\mathfrak{S}'_q$
stand for the reducts of $\mathfrak{S}$ and $\mathfrak{S}'$ to
$(\{1,2,3,4\}, \{ q\})$-stit models.

\begin{lemma}\label{L:b}
The relation $B := \{ (h_m, g_{m_1}) \mid (m, m_1 \in
4Tup),\,\&\,((\dag, h_m) \in V(q) \Leftrightarrow (\ddag, g_{m_1})
\in V'(q)) \}$ is a bisimulation between $(\mathfrak{S}_q, \dag)$
and $(\mathfrak{S}'_q, \ddag)$.
\end{lemma}
\begin{proof}
We first note that it follows from Definition \ref{D:choice}.6
that $(\ddag, g_{(0,0,0,0)^+}) \in V'(q)$ and $(\ddag,
g_{(0,0,1,0)^+}) \notin V'(q)$. Now if $m \in 4Tup$ then either
$(\dag, h_m) \in V(q)$ or $(\dag, h_m) \notin V(q)$. In the former
case, we get $h_m\mathrel{B}g_{(0,0,0,0)^+}$, in the latter case
we get $h_m\mathrel{B}g_{(0,0,1,0)^+}$. Therefore, by
\eqref{E:hists} and \eqref{E:hists-dag}, the domain of $B$ is $\{
h_m \mid m \in 4Tup \} = H^\mathfrak{S}_\dag$, as desired. As for
the counterdomain, we may argue in the same fashion, noting that
it follows from definition of $V$ that $(\dag, h_{(0,0,0,0)^+})
\in V(q)$ and $(\dag, h_{(0,1,0,1)^-}) \notin V(q)$. Thus, we also
get that the counterdomain of $B$ is $\{ g_m \mid m \in 4Tup \} =
H^{\mathfrak{S}'}_\ddag$.

The condition \eqref{atoms} from Definition \ref{bisimulation}
holds simply by definition of $B$. It remains to check the other
two conditions in this definition.

\emph{Condition \eqref{forth}}. Assume that $m_1, m_2, m_3 \in
4Tup$ and $1 \leq j \leq 4$ are such that we have both
$h_{m_1}\mathrel{B}g_{m_2}$ and $h_{m_3} \in
Choice^\dag_j(h_{m_1})$. We need to consider the following cases:

\emph{Case 1}. We have $(\dag, h_{m_1}) \in V(q) \Leftrightarrow
(\dag, h_{m_3}) \in V(q)$. Then note that we have both $g_{m_2}
\in Choice'^\ddag_j(g_{m_2})$ and $h_{m_3}\mathrel{B}g_{m_2}$, the
latter by definition of $B$.

\emph{Case 2}.  We have $(\dag, h_{m_1}) \in V(q)$, but $(\dag,
h_{m_3}) \notin V(q)$.

\emph{Case 2a}. We have $core(m_2) \neq (a,b,0,0)$ for all $a,b
\in \{ 0, 1\}$. Then we must have $(\ddag, g_{core(m_2)^-}) \notin
V'(q)$ so that $h_{m_3}\mathrel{B}g_{core(m_2)^-}$. On the other
hand, we have, by the identity of cores and Corollary
\ref{L:choice-cor}, that $g_{core(m_2)^-} \in
Choice'^\ddag_j(g_{m_2})$.

\emph{Case 2b}. We have $core(m_2) = (a,b,0,0)$ for some $a,b \in
\{ 0, 1\}$. 
Now, if $j \in \{ 1,2,4 \}$ we note that for $m_4 := (a,b,1,0)^+$
we have $g_{m_4} \in Choice'^\ddag_j(g_{m_2})$ and also $(\ddag,
g_{m_4}) \notin V'(q)$ so that $h_{m_3}\mathrel{B}g_{m_4}$. On the
other hand, if $j = 3$, then we set $m_4 := (a,b,0,1)^-$ and,
again, get $g_{m_4} \in Choice'^\ddag_j(g_{m_2})$ and also
$(\ddag, g_{m_4}) \notin V'(q)$ so that
$h_{m_3}\mathrel{B}g_{m_4}$.

\emph{Case 3}.  We have $(\dag, h_{m_1}) \notin V(q)$, but $(\dag,
h_{m_3}) \in V(q)$. Then, by $h_{m_1}\mathrel{B}g_{m_2}$, also
$(\ddag, g_{m_2}) \notin V'(q)$ which means that  $core(m_2) \neq
(a,b,0,0)$ for all $a,b \in \{ 0, 1\}$.

\emph{Case 3a}. We have, moreover, that $core(m_2) \neq (a,b,1,0)$
for all $a,b \in \{ 0, 1\}$. Then we must have $(\ddag,
g_{core(m_2)^+}) \in V'(q)$ so that
$h_{m_3}\mathrel{B}g_{core(m_2)^+}$. On the other hand, we have,
by the identity of cores and Corollary \ref{L:choice-cor}, that
$g_{core(m_2)^+} \in Choice'^\ddag_j(g_{m_2})$.

\emph{Case 3b}. We have $core(m_2) = (a,b,1,0)$ for some $a,b \in
\{ 0, 1\}$. Now, if $j \in \{ 1,2,4 \}$ we note that for $m_4 :=
(a,b,0,0)^+$ we have $g_{m_4} \in Choice'^\ddag_j(g_{m_2})$ and
also $(\ddag, g_{m_4}) \in V'(q)$ so that
$h_{m_3}\mathrel{B}g_{m_4}$. On the other hand, if $j = 3$, then
we set $m_4 := (a,b,1,1)^+$ and, again, get $g_{m_4} \in
Choice'^\ddag_j(g_{m_2})$ and also $(\ddag, g_{m_4}) \in V'(q)$ so
that $h_{m_3}\mathrel{B}g_{m_4}$.

\emph{Condition \eqref{back}}. Assume that $m_1, m_2, m_3 \in
4Tup$ and $1 \leq j \leq 4$ are such that we have both
$h_{m_1}\mathrel{B}g_{m_2}$ and $g_{m_3} \in
Choice'^\ddag_j(g_{m_2})$. We need to consider the following
cases:

\emph{Case 1}. We have $(\ddag, g_{m_2}) \in V'(q) \Leftrightarrow
(\ddag, g_{m_3}) \in V'(q)$. Then note that we have both $h_{m_1}
\in Choice^\dag_j(h_{m_1})$ and $h_{m_1}\mathrel{B}g_{m_3}$, the
latter by definition of $B$.

\emph{Case 2}.  We have $(\ddag, g_{m_2}) \in V'(q)$, but $(\ddag,
g_{m_3}) \notin V'(q)$.

\emph{Case 2a}. For all $a,b \in \{ 0,1 \}$, we have both $m_1
\neq (a,b,0,0)$ and $m_1 \neq (0,0,a,b)$. Then we must have
$(\dag, h_{core(m_1)^-}) \notin V(q)$ so that
$h_{core(m_1)^-}\mathrel{B}g_{m_3}$. On the other hand, we have,
by the identity of cores, that $h_{core(m_1)^-} \in
Choice^\dag_j(h_{m_1})$.

\emph{Case 2b}. We have $core(m_1) = (0,0,0,0)$.  Now, if $j \in
\{ 1,3 \}$, we note that for $m_4 := (0,1,0,1)^-$ we have $g_{m_4}
\in Choice^\dag_j(h_{m_1})$ and also $(\dag, h_{m_4}) \notin V(q)$
so that $h_{(0,1,0,1)^-}\mathrel{B}g_{m_3}$. On the other hand, if
$j \in \{ 2,4 \}$, then we set $m_4 := (1,0,1,0)^-$ and, again,
get $h_{(1,0,1,0)^-} \in Choice^\dag_j(h_{m_1})$ and also $(\dag,
h_{(1,0,1,0)^-}) \notin V(q)$ so that
$h_{(1,0,1,0)^-}\mathrel{B}g_{m_3}$.

\emph{Case 2c}. We have $core(m_1) = (0,0,a,b)$ for some $a,b
 \in \{ 0,1 \}$ such that $(a,b) \neq (0,0)$. Then we have to instantiate $j$:

For $j = 1$, we set $m_4 := (0,1,a,b)^-$.

For $j \in \{ 2,3,4 \}$, we set $m_4 := (1,0,a,b)^-$.

Under these settings, we always get both $h_{m_4} \in
Choice^\dag_j(h_{m_1})$ for the respective $j$, and $(\dag,
h_{m_4}) \notin V(q)$ so that $h_{m_4}\mathrel{B}g_{m_3}$.

\emph{Case 2d}. We have $core(m_1) = (a,b,0,0)$ for some $a,b \in
\{ 0,1 \}$  such that $(a,b) \neq (0,0)$. Then we have to
instantiate $j$:

For $j \in \{ 1,2,3 \}$, we set $m_4 := (a,b,0,1)^-$.

For $j = 4$, we set $m_4 := (a,b,1,0)^-$.

Under these settings, we always get both $h_{m_4} \in
Choice^\dag_j(h_{m_1})$ for the respective $j$, and $(\dag,
h_{m_4}) \notin V(q)$ so that $h_{m_4}\mathrel{B}g_{m_3}$.

\emph{Case 3}. We have $(\ddag, g_{m_2}) \notin V'(q)$, but
$(\ddag, g_{m_3}) \in V'(q)$. Then we must have $(\dag,
h_{core(m_1)^+}) \in V(q)$ so that
$h_{core(m_1)^+}\mathrel{B}g_{m_3}$. On the other hand, we have,
by the identity of cores  and Corollary \ref{L:choice-cor}, that
$h_{core(m_1)^+} \in Choice^\dag_j(h_{m_1})$.
\end{proof}
We are now in a position to prove Theorem \ref{negative}.

\begin{proof}[Proof of Theorem \ref{negative}] Assume for
\emph{reductio}, that stit logic has $(RCIP)_n$ for some $n > 3$.
Then $n \geq 4$ and both $A: = \Diamond([1]p \wedge [2](p \to q))$
and $B: = \neg\Diamond([3]r \wedge [4](r \to \neg q))$ are in
$\mathcal{L}^{\{1,\ldots, n\}}_{\{ p,q,r \}}$. By Lemma
\ref{counterexample}, we have $\vdash A \to B$, therefore, by
Definition \ref{D:rcip}, there must be a $C \in
\mathcal{L}^{\{1,2,3,4\}}_{\{ q \}}$ such that both $\vdash A \to
C$ and $\vdash C \to B$. We choose such a $C$ and note that, by
Lemma \ref{satisfied}, we have $\mathfrak{S}, \dag,
h_{(0,0,0,0)^+} \models A$, therefore, by $\vdash A \to C$ and the
strong completeness of $\mathbb{S}$ w.r.t. stit logic, we must
also have $\mathfrak{S}, \dag, h_{(0,0,0,0)^+} \models C$. The
latter means that, moreover, $\mathfrak{S}_q, \dag,
h_{(0,0,0,0)^+} \models C$, since $C \in
\mathcal{L}^{\{1,2,3,4\}}_{\{ q \}}$. Note that it follows from
the definition of $B$ as given in Lemma \ref{L:b} that
$h_{(0,0,0,0)^+}\mathrel{B}g_{(0,0,0,0)^+}$, therefore, it follows
from Lemmas \ref{L:b} and \ref{L:bisimulation} that also
$\mathfrak{S}'_q, \ddag, g_{(0,0,0,0)^+} \models C$. Again, by the
fact that $C \in \mathcal{L}^{\{1,2,3,4\}}_{\{ q \}}$, we infer
that $\mathfrak{S}', \ddag, g_{(0,0,0,0)^+} \models C$, whence it
follows by $\vdash C \to B$, that we must also have
$\mathfrak{S}', \ddag, g_{(0,0,0,0)^+} \models B$. But the latter
is in contradiction with Lemma \ref{satisfied} which says that, on
the contrary, $\mathfrak{S}', \ddag, g_{(0,0,0,0)^+} \not\models
B$. So we have got our contradiction in place.
\end{proof}

\section{Further developments and ramifications}\label{S:further}

The main topic of this paper is the Restricted Interpolation
Property as given by Definition \ref{D:rcip}. This property is
much weaker than the simple Craig Interpolation Property which has
attracted much more attention in the existing literature, and for
a good reason. In the context of stit logic, we may formulate the
Craig Interpolation Property as follows:
\begin{definition}\label{D:craig}
Stit logic has the $n$-\emph{Craig Interpolation Property}
(abbreviated by $(CIP)_n$) iff for any set of propositional
variables $V$, and all $A, B \in \mathcal{L}^{\{1,\ldots, n\}}_V$,
whenever $\vdash A \to B$, then there exists a $C \in
\mathcal{L}^{Ag(A) \cup Ag(B)}_{|A| \cap |B|}$ such that both
$\vdash A \to C$ and $\vdash C \to B$.
\end{definition}
Then the relevance of the above results to this latter much more
important version of interpolation can be summed up in two
following corollaries:
\begin{corollary}\label{C:craig-restricted}
For all positive integers $n$, if stit logic does not have
$(RCIP)_n$, then stit logic does not have $(CIP)_n$.
\end{corollary}
\begin{proof}
Immediately from Definition \ref{D:rcip} and Definition
\ref{D:craig}.
\end{proof}
\begin{corollary}\label{C:craig}
For all $n > 3$, stit logic does not have $(CIP)_n$.
\end{corollary}
\begin{proof}
Immediately from Corollary \ref{C:craig-restricted} and Theorem
\ref{negative}.
\end{proof}
Thus we may infer from the results of the above sections that stit
logic fails $(CIP)_n$ for \emph{almost} all positive integers $n$.
The failure of $(CIP)_n$ further entails, by the standard
argument, the failure of the Robinson Consistency Property for the
respective values of $n$. Furthermore, Theorem \ref{positive}
allows us to considerably limit our search for counterexamples to
$(CIP)_n$ for the remaining few values of $n$. Namely, it follows
from Theorem \ref{positive} that whenever $\vdash A \to B$ does
not have an interpolant in the sense of Definition \ref{D:craig},
then we must have $Ag(A) \cap Ag(B) \neq \emptyset$.

Turning again to the Robinson Consistency Property and its
variants, Definition \ref{D:rcip} raises a natural question
whether $(RCIP)_n$ has its accompanying restricted version of the
Robinson Consistency Property. The answer is yes, and the
respective version of the Robinson Consistency Property can be
formulated as follows:
\begin{definition}\label{D:robinson}
Stit logic has the \emph{Restricted} $n$-\emph{Robinson
Consistency Property} (abbreviated by $(RRCP)_n$) iff for any set
of propositional variables $V$, and all $\Gamma, \Delta \subseteq
\mathcal{L}^{\{1,\ldots, n\}}_V$, if $(\Gamma, \Delta)$ is
inseparable, then $\Gamma \cup \Delta$ is consistent.
\end{definition}
On the basis of this definition and the proofs given in Sections
\ref{S:positive} and \ref{S:negative}, the following theorem can
be established:
\begin{theorem}\label{T:robinson}
For every positive integer $n$, stit logic has $(RRCP)_n$ iff it
has $(RCIP)_n$.
\end{theorem}
\begin{proof}[Proof (a sketch)]
By a standard argument, one can show that whenever stit logic
fails $(RCIP)_n$, it also fails $(RRCP)_n$. In the other
direction, an obvious modification of the proof of Theorem
\ref{positive} given above shows that stit logic has $(RRCP)_n$
for all $n \leq 3$.
\end{proof}
Finally, we tackle the question of the logical status of action
modalities. Definition \ref{D:rcip} treats action modalities of
the form $[j]$ for a $j \in Ag$ as logical symbols, and this is in
accordance with the standard view of modalities. But it is easy to
see that one can also argue in favor of non-logical status of
these modalities, since the agent indices are often treated as
proper names of respective agents, and proper names are
non-logical. If this attitude is carried out systematically, then
we get the following strengthening of Definition \ref{D:rcip}:
\begin{definition}\label{D:strongrestricted}
Stit logic has the \emph{Strong Restricted }$n$-\emph{Craig
Interpolation Property} (abbreviated by $(SRCIP)_n$) iff for any
set of propositional variables $V$, and all $A, B \in
\mathcal{L}^{\{1,\ldots, n\}}_V$, whenever $\vdash A \to B$ and
$Ag(A) \cap Ag(B) = \emptyset$, then there exists a $C \in
\mathcal{L}^{\emptyset}_{|A| \cap |B|}$ such that both $\vdash A
\to C$ and $\vdash C \to B$.
\end{definition}
One immediately sees that $(SRCIP)_n$ only differs from $(RCIP)_n$
in placing stricter requirements on the interpolant. Therefore,
for any given positive integer $n$, the failure of $(RCIP)_n$ for
stit logic entails the failure of $(SRCIP)_n$ so that it follows
from Theorem \ref{negative} that stit logic fails $(SRCIP)_n$ for
all positive integers $n > 3$. This result, however, can be
improved as follows:
\begin{theorem}\label{strong-negative}
For every $n > 1$, stit logic does not have $(SRCIP)_n$.
\end{theorem}
In order to prove this theorem, we again need to establish a
number of technical claims:
\begin{lemma}\label{s-counterexample}
Let $j_1, j_2 \in Ag$ be different and let $p$ be a propositional
variable. Then:
$$
\vdash \Diamond[j_1]p \to \neg\Diamond[j_2]\neg p.
$$
\end{lemma}
\begin{proof}
We reason as follows:
\begin{align}
    &(\Diamond[j_1]p \wedge \Diamond[j_2]\neg p) \to \Diamond([j_1]p \wedge [j_2]\neg
    p)\label{E:sce1}&&\text{(by \eqref{A3})}\\
&([j_1]p \wedge [j_2]\neg p) \to
\bot\label{E:sce2}&&\text{($[j_1]$, $[j_2]$ are S5)}\\
&\Diamond([j_1]p \wedge [j_2]\neg p) \to \bot\label{E:sce3}&&\text{(from \eqref{E:sce2} since $\Box$ is S5)}\\
 &(\Diamond[j_1]p \wedge \Diamond[j_2]\neg p) \to \bot\label{E:sce4}&&\text{(from \eqref{E:sce1} and \eqref{E:sce3})}
\end{align}
\end{proof}
\begin{lemma}\label{L:s-bisimulation}
Let $\mathfrak{S} = \langle Tree, \leq, Choice, V\rangle$ and
$\mathfrak{S}' = \langle Tree', \leq', Choice', V'\rangle$ be an
$(Ag, V)$-stit model and an $(Ag', V)$-stit model, respectively,
and let $m \in Tree$ and $m' \in Tree'$. Let relation $B \subseteq
H^\mathfrak{S}_m \times H^{\mathfrak{S}'}_{m'}$ be such that the
domain of $B$ is $H^\mathfrak{S}_m$, the counter-domain of $B$ is
$H^{\mathfrak{S}'}_{m'}$, and assume that $B$ satisfies condition
\eqref{atoms}. Then, whenever $A \in \mathcal{L}^\emptyset_V$, we
will have, for all $h_1 \in H^\mathfrak{S}_m$ and $h'_1 \in
H^{\mathfrak{S}'}_{m'}$:
$$
h_1\mathrel{B}h'_1 \Rightarrow (\mathfrak{S}, m, h_1 \models A
\Leftrightarrow \mathfrak{S}', m', h'_1 \models A).
$$
\end{lemma}
\begin{proof}
We reason in the same way as in the proof of Lemma
\ref{L:bisimulation}, the only difference being that Case 2 in the
induction step can be omitted.
\end{proof}
We are now in a position to prove Theorem \ref{strong-negative}.
\begin{proof}[Proof of Theorem \ref{strong-negative}]
Consider the following sets and structures:
\begin{itemize}
\item $Tr = \{ m, m_0, m_1 \}$.

\item $\unlhd$ is the reflexive closure of the relation $\{ (m,
m_0), (m, m_1) \}$.
\end{itemize}
The two histories induced by $(Tr, \unlhd)$ are $h_0 = \{ m, m_0
\}$ and $h_1 = \{ m, m_1 \}$. We now define two further sets:
\begin{itemize}
\item $U = \{ (m, h_0) \}$.

\item $F = \{ (m, \{\{h_0\},\{h_1\}\}), (m_0, \{\{h_0\}\}), (m_1,
\{\{h_1\}\})\}$.
\end{itemize}
It is immediate to establish that the structure $\mathfrak{M}_{j,
p} = (Tr, \unlhd, F_j, U_p)$, in which if $F_j$ interprets $F$ as
the choice function for a given single agent $j$ and $U_p$
interprets $U$ as the evaluation for a given single propositional
variable $p$, is a $(\{j\},\{p\})$-stit structure.

We now consider two stit models, $\mathfrak{M}_{1, p}$ and
$\mathfrak{M}_{2, p}$, and we set $B$ as the diagonal of $Hist(Tr,
\unlhd)$, in other words, we set $B: = \{(h_0,h_0), (h_1,h_1)\}$.
It is clear that $B$ satisfies the conditions of Lemma
\ref{L:s-bisimulation} so that for every $C \in
\mathcal{L}^\emptyset_{\{p\}}$ which contains no action
modalities, we will have:
\begin{equation}\label{E:new}
    \mathfrak{M}_{1, p}, m, h_0 \models C \Leftrightarrow \mathfrak{M}_{2, p}, m, h_0 \models
    C.
\end{equation}
Now assume that  $(SRCIP)_n$ holds for any $n$ greater than one.
We will show that this assumption leads to a contradiction.
Indeed, it follows then from Lemma \ref{s-counterexample} that
there must be a formula $C \in \mathcal{L}^\emptyset_{\{p\}}$ such
that the following holds:
\begin{align}
    &\vdash \Diamond[1]p \to C\label{E:new1}\\
    &\vdash C \to \neg\Diamond[2]\neg p\label{E:new2}
\end{align}
Choose any such $C$. We obviously have $\mathfrak{M}_{1, p}, m,
h_0 \models \Diamond[1]p$ so that it follows from \eqref{E:new1}
and the soundness of $\mathbb{S}$ that $\mathfrak{M}_{1, p}, m,
h_0 \models C$, whence, by \eqref{E:new}, also $\mathfrak{M}_{2,
p}, m, h_0 \models C$. From the latter, together with
\eqref{E:new2}, it follows that we should have $\mathfrak{M}_{2,
p}, m, h_0 \models \neg\Diamond[2]\neg p$, whereas the direct
check shows that we in fact have $\mathfrak{M}_{2, p}, m, h_0
\models \Diamond[2]\neg p$. Thus we have got our contradiction in
place.
\end{proof}
The Strong Restricted Craig Interpolation Property admits of the
following unrestricted companion:
\begin{definition}\label{D:strongcraig}
Stit logic has the \emph{Strong }$n$-\emph{Craig Interpolation
Property} (abbreviated by $(SCIP)_n$) iff for any set of
propositional variables $V$, and all $A, B \in
\mathcal{L}^{\{1,\ldots, n\}}_V$, there exists a $C \in
\mathcal{L}^{Ag(A) \cap Ag(B)}_{|A| \cap |B|}$ such that both
$\vdash A \to C$ and $\vdash C \to B$.
\end{definition}
Of course, for a given positive integer $n$, $(SCIP)_n$ is at
least as strong as $(SRCIP)_n$, whence we get the following
corollary to Theorem \ref{strong-negative}:
\begin{theorem}\label{strong-craig-negative}
For every $n > 1$, stit logic does not have $(SCIP)_n$.
\end{theorem}

\section{Conclusion}\label{conclusion}
In the preceding text, we have looked into the question of whether
stit logic has the Restricted $n$-Craig Interpolation Property,
showing that the answer is in the affirmative iff $n \leq 3$. We
have also briefly looked into some related properties, showing
that the Restricted Craig Interpolation for stit logic has its
natural accompanying version of the Robinson Consistency Property
which turns out to be equivalent to the Restricted Craig
Interpolation for every positive integer $n$. From these results,
we have drawn the corollary that the unrestricted $n$-Craig
Interpolation fails for stit logic under every instantiation of $n
> 3$, that is to say, for almost all positive integers $n$. We
have also shown that if one treats action modalities as
non-logical symbols, the scope of interpolation failures extends
to include the case when $n \in \{2, 3\}$, and this extension
occurs for the strengthened versions of both unrestricted and
restricted $n$-Craig Interpolation Property.

The import of this almost universal failure of Craig Interpolation
for stit logic can be seen sharper if one takes into an account
that the axiomatic system $\mathbb{S}$ for this logic, as given in
Section \ref{preliminaries} above, suggests that stit logic is an
extension of propositional multi-S5. It is a well-known fact, see
e.g. \cite{vB1997}, that multi-S5 has the Craig Interpolation
Property.\footnote{In fact, multi-S5 even enjoys strong
interpolation in the sense that one may demand that only shared S5
modalities occur in the interpolant for a given valid
implication.} Thus the fact that this property fails for stit
logic highlights the fact that the difference between multi-S5 and
stit logic is quite substantial. Another conclusion is that, in
extending multi-S5, stit logic upsets the delicate balance between
deductive power and expressivity which is present in multi-S5.

As the main problem for the future research remains the question
whether unrestricted $n$-Craig Interpolation Property holds for
all or at least some $n \leq 3$ and whether the natural Robinson
Consistency companions of the $n$-Craig Interpolation Property can
be distinguished from this property on this, rather limited, set
of values.

\section{Acknowledgements}
To be inserted.

}


\begin{thebibliography}{1}

\bibitem{balbiani}
P.~Balbiani, A.~Herzig, and E.~Troquard.
\newblock Alternative axiomatics and complexity of deliberative stit theories.
\newblock {\em Journal of Philosophical Logic}, 37(4):387--406, 2008.


\bibitem{belnap2001facing}
N.~Belnap, M.~Perloff, and M.~Xu.
\newblock {\em Facing the Future: Agents and Choices in Our Indeterminist
  World}.
\newblock Oxford University Press, 2001.

\bibitem{broersen}
J.~Broersen.
\newblock Deontic epistemic stit logic distinguishing modes of mens rea.
\newblock {\em Journal of Applied Logic}, 9(2):137--152, 2011.

\bibitem{chellas}
B.~Chellas.
\newblock {\em The Logical Form of Imperatives}.
\newblock Perry Lane Press, Stanford, CA, 1969.

\bibitem{gabbay}
D.~Gabbay and L.~Maksimova.
\newblock {\em Interpolation and Definability: Modal and Intuitionistic Logics}.
\newblock Oxford University Press, 2005.

\bibitem{HerzigSchwarzent}  A.~Herzig and F.~Schwarzentruber.
\newblock 'Properties of logics of individual
and group agency', in: C.~Areces and R.~Goldblatt (eds.), {\em
Advances in Modal Logic, Volume 7} , College Publications, London,
133--149, 2008.

\bibitem{horty2001agency}
J.~Horty.
\newblock {\em Agency and Deontic Logic}.
\newblock Oxford University Press, USA, 2001.

\bibitem{horty1995}
J.~Horty and N.~Belnap.
\newblock The deliberative stit: a study of action, omission, ability and obligation.
\newblock {\em Journal of Philosophical Logic}, 24:583--644, 1995.

\bibitem{lorini}
E.~Lorini.
\newblock `Temporal stit logic and its application to
normative reasoning. \newblock {\em Journal of Applied
Non-Classical Logics}, 23(4):372--399, 2013.

\bibitem{OLWA}
G.~Olkhovikov and H.~Wansing.
\newblock Inference as doxastic agency. Part I: The basics of justification
  stit logic.
  \newblock {\em Studia Logica}, Online first: https://doi.org/10.1007/s11225-017-9779z, 2018.

\bibitem{OLWA2}
G.~Olkhovikov and H.~Wansing.
\newblock Inference as doxastic agency. Part II: Ramifications and refinements.
\newblock {\em Australasian Journal of Logic}, 14(4):408–438, 2017.

\bibitem{vB1997}
J.~van~Benthem.
\newblock Modal foundations for predicate logic.
\newblock {\em Logic Journal of the IGPL}, 5(2):259--286, 1997.

\bibitem{vK1986}
F.~von~Kutschera.
\newblock Bewirken.
\newblock {\em Erkenntnis}, 24:253--281, 1986.
\end{thebibliography}
\end{document}